\documentclass[12pt,a4paper, twoside]{amsart}
\usepackage{amsfonts}
\usepackage{amssymb}
\usepackage[english]{babel}
\usepackage{mathtools}
\usepackage{amscd}
\usepackage{amsthm}
\usepackage{thmtools}
\usepackage{nicefrac}
\usepackage{amscd}
\usepackage{tikz-cd}
\usepackage{tikz}
\usepackage{enumerate}
\usetikzlibrary{matrix,arrows,decorations.pathmorphing}
\usepackage{setspace}
\usepackage{hyperref}

\newcommand{\partialt}{\frac{\partial}{\partial t}}
\newcommand{\partialb}{\bar \partial}
\newcommand{\Ric}{\operatorname{Ric}}
\newcommand{\OPE}{\mathcal{O}_{\mathbb{P}(E)}(-1)}

\makeatletter
\def\paragraph{\@startsection{paragraph}{4}%
	\z@\z@{-\fontdimen2\font}%
	{\normalfont\bfseries}}
\makeatother

\makeatletter
\newcommand{\definetitlefootnote}[1]{%
	\newcommand\addtitlefootnote{%
		\makebox[0pt][l]{$^{*}$}%
		\footnote{\protect\@titlefootnotetext}
	}%
	\newcommand\@titlefootnotetext{\spaceskip=\z@skip $^{*}$#1}%
}
\makeatother

\begin{document}
	
	\newtheorem{lemma}{Lemma}[section]
	\newtheorem{prop}[lemma]{Proposition}
	
	\newtheorem{theorem}[lemma]{Theorem}
	\newtheorem{corollary}[lemma]{Corollary}
	\newtheorem{theoremintro}{Theorem}
	\theoremstyle{definition}
	\newtheorem{defi}[lemma]{Definition}
	\newtheorem{example}[lemma]{Example}
	\newtheorem*{notation}{Notation}
	\newtheorem*{claim}{Claim}
	
	\newtheorem{remark}[lemma]{Remark}
	
%\pagenumbering{gobble}

	\title[Existence and Uniqueness of K\"ahler-Ricci solitons]{Existence and Uniqueness of $S^1$-invariant K\"ahler-Ricci solitons} 
	
	\author{Johannes Sch\"afer}

	\begin{abstract}
		We use the momentum construction for $S^1$-invariant K\"ahler metrics as developed by Hwang-Singer to construct new examples of steady K\"ahler-Ricci solitons.  We also prove that these solitons are unique in their K\"ahler class, provided the vector field and the asymptotic behavior are fixed. 
	\end{abstract}
	
	\maketitle
	
	%\addtocounter{page}{12}

	\section{Introduction}

A \textit{steady K\"ahler-Ricci soliton} is a K\"ahler manifold $(M,g)$ whose K\"ahler form $\omega$ satisfies
\begin{align}\label{soliton equation}
	\Ric(\omega) = - \mathcal{L}_X \omega 
\end{align}
for some  vector field $X$ which is the real part of a holomorphic vector field.  Solutions to (\ref{soliton equation}) are natural generalizations of Ricci-flat metrics and arise as self-similar solutions to Ricci flow.

If the vector field $X$ is non-zero, the manifold must be non-compact \cite{ivey1993ricci}. In general,  there is no classification for steady K\"ahler-Ricci solitons available and only few examples are known. Even if a manifold admits a K\"ahler-Ricci  soliton, it is  not understood which subset of the K\"ahler cone contains further examples of Ricci solitons.  It is also not clear, how many solitons there are in each K\"ahler class.

All known examples with $X\neq 0$ are divided into two classes. One class contains explicitly constructed solutions by using ODE methods  (\cite{hamilton1988ricci}, \cite{Caosoliton}, \cite{chave1996class}, \cite{pedersen1999quasi}, \cite{feldman2003rotationally}, \cite{li2010rotationally}, \cite{dancer2011ricci},\cite{futaki2011constructing}, \cite{yang2012characterization}), while the other examples are obtained by using PDE gluing methods (\cite{Biquard17}). The explicit examples are constructed  on Euclidean space or on holomorphic vector bundles over K\"ahler manifolds, while the gluing method produces solitons on certain crepant resolutions of  orbifolds $\mathbb{C}^n /G$. 

In this article,  we use the momentum construction introduced by Hwang-Singer \cite{hwang2002momentum} to find new examples of steady K\"ahler-Ricci solitons.
More precisely, we prove the following theorem. 
\begin{theorem} \label{existence theorem line bundles}
	Let $\pi: K_M \to (M, g_M)$ be   the canonical line bundle over a compact K\"ahler manifold. Assume  that the Ricci form  of $g_M$ is positive semi-definite and has constant eigenvalues  with respect to $g_M$. Then $K_M$ admits a 1-parameter family of complete steady  K\"ahler-Ricci solitons in the K\"ahler class $[\pi^*\omega_M]$.
\end{theorem}

Theorem \ref{existence theorem line bundles} generalises results obtained in \cite{Caosoliton}, \cite{chave1996class}, \cite{pedersen1999quasi}, \cite{feldman2003rotationally}, \cite{dancer2011ricci} and \cite{yang2012characterization}. The main difference is that we do \textit{not} assume $(M,g_M)$ to be a K\"ahler-Einstein Fano manifold, but only require that $\operatorname{Ric}(\omega_M)$ has constant eigenvalues.

Under the same assumption, Hwang-Singer \cite{hwang2002momentum} used Calabi's ansatz  to construct K\"ahler-Einstein metrics on line bundles. They observed that the constancy of eigenvalues is sufficient to reduce the K\"ahler-Einstein equation to a single ODE, which is linear after applying a certain transformation. We prove Theorem \ref{existence theorem line bundles} by adapting their construction to the case of steady K\"ahler-Ricci solitons.

Theorem \ref{existence theorem line bundles} produces new examples if the base $M$ is a flag variety. More concretely, consider  the canonical bundle over $M=\mathbb{P}(T^*\mathbb{CP}^n)$, the projectivization of the cotangent bundle $T^*\mathbb{CP}^n$. Previously, it was only known that compactly supported K\"ahler classes admit steady solitons (\cite{pedersen1999quasi}, \cite{dancer2011ricci}, \cite{yang2012characterization}), whereas Theorem \ref{existence theorem line bundles} shows  they sweep out the \textit{entire} K\"ahler cone.

Another interesting feature of Hwang-Singer's construction is that it can also be applied to certain vector bundles of rank $\geq2$. Then we obtain a result analogue to Theorem \ref{existence theorem line bundles}. 

\begin{theorem} \label{existence theorem VB}
	Let $\pi : E \to D$ be a holomorphic vector bundle of rank $m$ over a compact K\"ahler manifold $(D,\omega_D)$. Assume that $E$ admits a Hermitian metric $h$ such that the corresponding  curvature form $\gamma$ of the tautological bundle $(\mathcal{O}_{\mathbb{P}(E)} (-1),h )$ is negative semi-definite and has constant eigenvalues with respect to the K\"ahler metric $\omega_M= p^*\omega_D -\gamma$, where $p:M=\mathbb{P}(E) \to D$ is the natural projection. Additionally, suppose that 
	\begin{align} \label{topological restriction on gamma and ricM}
	\Ric(\omega_M) = - m \gamma.
	\end{align}
	Then $E$ admits a 1-parameter family of complete steady K\"ahler-Ricci solitons in the class $[\pi^* \omega_D]$. 
\end{theorem}

This can be applied to certain sums of line bundles and again, if the base is a flag variety, it constructs steady solitons in each K\"ahler class,  generalising  results in \cite{li2010rotationally} and \cite{dancer2011ricci}[Theorem 4.20].

Given a K\"ahler-Ricci soliton, it is an interesting question  whether or not it is unique in its K\"ahler class. It is natural to fix a vector field for this question because there can be families of solitons as in Theorem 1.1 and 1.2 for instance. In general, this question seems to be largely open.

In the special case of  Ricci-flat K\"ahler metrics, the  question of uniqueness is studied under additional assumptions on the asymptotic behaviour of the metric (\cite{joyce2000compact}, 
\cite{goto2012calabi}, \cite{HAJOricciflat}, \cite{haskins2015asymptotically}).  For example,  asymptotically conical Ricci-flat  metrics are unique in their K\"ahler class  \cite{HAJOricciflat}.

In the different setting of solitons with $X\neq 0$, there are only few results such as \cite{Biquard17}. Assuming that two steady solitons $\omega_1,\omega_2$ with the same vector field are related by $\omega_1= \omega_2 + \sqrt{-1} \partial \partialb u$, \cite{Biquard17}[Proposition 1.2] shows that $\omega_1=\omega_2$ provided $u$ and its derivatives tend to zero at infinity. 

In this work, we  extend the previous result for the metrics constructed in Theorem \ref{existence theorem line bundles} and \ref{existence theorem VB}. 

\begin{theorem}\label{uniqueness theorem in intro}
Let $E\to D$ be a holomorphic vector bundle satisfying the assumptions in Theorem \ref{existence theorem line bundles} or \ref{existence theorem VB} and denote  the steady K\"ahler-Ricci solitons constructed in Theorem \ref{existence theorem line bundles} or \ref{existence theorem VB} by $\omega_{\varphi}$. Suppose that $\omega$ is a K\"ahler-Ricci soliton on $E$ with the same vector field as $\omega_{\varphi}$ such that  $[\omega]=[\omega_{\varphi}] \in H^2 (E)$. If moreover $\omega_{\varphi}-\omega \in C_{-\delta}^{\infty}(\Lambda^2 T^*E)$ for some $\delta>2$, then $\omega_{\varphi}= \omega$. 	
\end{theorem}

We reduce the proof of Theorem \ref{uniqueness theorem in intro} to \cite{Biquard17}[Proposition 1.2] by proving a $\partial \partialb$-Lemma with controlled growth. Assuming that $\omega_{\varphi}-\omega$ is asymptotic to zero, in a suitable sense, we show that there exists a smooth function $u $ such that $\omega_{\varphi}-\omega =  \sqrt{-1} \partial \partialb u$ and $u\in C^{\infty}_{-\delta+2 }(E)$, i.e. $u$ and all its derivatives tend to zero because $2-\delta<0$.

The strategy for finding such a function $u$ is analogue to \cite{HAJOricciflat}[Section 3]. The main point is proving that all harmonic 1-forms of a certain growth behaviour are identically zero which  requires   non-negative Ricci curvature. We will see that this is indeed true for the metrics $\omega_{\varphi}$ constructed in Theorem \ref{existence theorem line bundles} and \ref{existence theorem VB}.

This article is structured as follows. In Section 2, we recall Hwang-Singer's construction of K\"ahler metrics and prove Theorem \ref{existence theorem line bundles}. For proving Theorem \ref{existence theorem VB}, we have to make some adjustments  which are explained in Section 3. The metrics are studied more closely in Section 4. Here we observe in particular that the curvature of these metrics is bounded and that the Ricci curvature is non-negative. Then, in Section 5, we  prove Theorem \ref{uniqueness theorem in intro} by studying the Laplace operator and 
 harmonic 1-forms of the metrics constructed in Theorem \ref{existence theorem line bundles} and \ref{existence theorem VB}. 
 
 After the first version of this paper was uploaded to the arXiv, Conlon and Deruelle \cite{conlon2020steady} posted a preprint on the arXiv containing a new existence result for steady K\"ahler-Ricci solitons. There is some overlap between their main result \cite{conlon2020steady}[Theorem A] and our Theorems \ref{existence theorem line bundles} and  \ref{existence theorem VB}, compare Remarks \ref{Remak I on Conlon and Deruelle} and \ref{Remak II on Conlon and Deruelle} below.

\subsection*{Acknowledgement}
This article is part of the author's PhD thesis. The author is financially supported  by the graduate school \grqq IMPRS on Moduli Spaces" of the Max-Planck-Institute for Mathematics in Bonn  and  would like to thank his advisor, Prof. Ursula Hamenstädt, for her encouragement as well as helpful discussions. Moreover, the author is grateful to Prof. Hans-Joachim Hein for his interest in this work and his comments on earlier versions of this article.

	\section{Calabi's Ansatz for line bundles} \label{section calabis ansatz for line bundles}

Hwang-Singer's construction  combines Calabi's ansatz with ideas from symplectic geometry (\cite{hwang2002momentum}).  If $\pi: (L,h)\to (M,\omega_M)$ denotes a Hermitian holomorphic  line bundle over a K\"ahler manifold, then Calabi's idea (\cite{calabi1979metriques}) was to  search for K\"ahler metrics of the form
\begin{align}\label{calabi ansatz in introduction}
\pi ^* \omega_M + \sqrt{-1} \partial \partialb f(t).
\end{align}
Here,  $t$ denotes the logarithm of the fibre-wise norm function induced by $h$ and  $f$ is a convex function of one variable. Instead of describing the metric (\ref{calabi ansatz in introduction}) in terms of the potential $f$,  Hwang-Singer  introduced a new variable $\tau=\tau(t)$ and a function $\varphi=\varphi(\tau): (0,\infty) \to \mathbb{R}_+$ which is related to the Legendre transformation $F$ of $f$ by $\varphi= 1/F''$. In particular, $\varphi$ determines the metric (\ref{calabi ansatz in introduction}) uniquely.

Assuming that the curvature form of $h$ has constant eigenvalues, we will see in this section that the non-linear K\"ahler-Ricci soliton equation (\ref{soliton equation}) is equivalent to a single, linear ODE in the function $\varphi$, which can be solved explicitly. This leads to a proof of Theorem \ref{existence theorem line bundles}. Additionally, we discuss the main examples to which Theorem \ref{existence theorem line bundles} applies. 

\subsection{Notation and set-up.} We begin by briefly recalling Calabi's construction of K\"ahler metrics in the special case of the canonical bundle.  We follow the presentation in \cite{hwang2002momentum}[Section 2].

Let $(M^n,\omega_M)$ be a K\"ahler manifold of complex dimension $n$ and equip  its canonical  line bundle $\pi: K_M \to M$ with the Hermitian metric $h$ induced by $\omega_M$. Let $\gamma $ be the curvature form of  $h$ and assume that $-\gamma\geq 0$, i.e. $\gamma$ is negative semi-definite. Recall that  $\gamma$ is  given by
\[
\gamma= - \sqrt{-1} \partial \partialb \log h(s,\bar{s})= - \operatorname{Ric}(\omega_M),
\]
where $s:U\to K_M$ is a local holomorphic section of $K_M$ and $\operatorname{Ric}(\omega_M) $ denotes the Ricci form of $\omega_M$.
We introduce the radial function $r: K_M \to \mathbb{R}_{\geq0}$ defined by $r(v)= \sqrt{h(v,\bar{v})}$ and outside the zero section, we define a new function $t:K_M \setminus M \to \mathbb{R}$ by $t=2 \log r$.  The pullback $\pi^*\gamma$ is a $\partial\partialb$-exact form on $K_M\setminus M$ and satisfies
\begin{align}\label{pullback of curvature form}
\pi^* \gamma = - \sqrt{-1} \partial \partialb t. 
\end{align}

 Suppose $f:\mathbb{R} \to \mathbb{R}$ is a smooth function satisfying
  \begin{align}\label{positivity assumption on f}
  \lim _{t \to - \infty} f'(t) =0 \; \text{ and} \; \; f''>0.
  \end{align}
  Then Calabi's Ansatz searches for K\"ahler metrics $\omega$ of the form
  \begin{align} \label{spelling out calabis ansatz}
  \omega=\pi^* \omega_M + \sqrt{-1} \partial \partialb f (t) =\pi ^* \omega_M -f'(t) \pi^* \gamma + f''(t) \sqrt{-1} \partial t \wedge \partialb t.
  \end{align}
  Note that $\omega$ is defined on $K_M \setminus M$, the canonical bundle with the zero section removed, and it is positive since we assumed $-\gamma\geq0$ and (\ref{positivity assumption on f}). Depending on the behaviour of $f(t)$ as $t\to \pm \infty$,  $\omega$ can be extended to all of $K_M$ and define a complete metric. When this can happen is explained in  the next subsection.

    We conclude this subsection by describing the Calabi metric $\omega$ in terms of the Legendre transformation  of its potential $f$, which is well-defined since $f$ is convex by (\ref{positivity assumption on f}).  
   We  now briefly recall this transformation.
  Let $I= \operatorname{Im} f' \subset \mathbb{R}_{+}$ be the image of $f'$ and define the new variable $\tau:= f'(t) \in I$. We write $I=(0, \tau_2)$, which  means that
  \begin{align*}
  \lim_{s\to -\infty} \tau(s) = \lim_{t\to -\infty} f'(t) = 0 , \;\;\; \lim_{s\to + \infty } \tau(s) = \lim_{t\to + \infty} f'(t) = \tau_2.
  \end{align*}
We point out that in general $\tau_2\leq +\infty$, but in the case considered in subsequent sections,  we have in fact that $\tau_2=+\infty$. 
  The Legendre transform $F: I \to \mathbb{R} $ is defined by the formula
  \begin{align*}
  f(t) + F(\tau) =t \tau.
  \end{align*}
  One can check that $F$  is also strictly convex, so that we can define  a new function $\varphi: I\to \mathbb{R}_{+}$ by 
  \begin{align*}
  \varphi(\tau) = \frac{1}{F''(\tau)}.
  \end{align*}
  Then we obtain the following relations
  \begin{align}\label{new relations after transformation}
  \frac{d\tau}{dt}= f''(t)=\varphi(\tau), \; \;\;  f'''(t)= \frac{d \varphi}{dt}  = \varphi'(\tau ) \varphi(\tau).
  \end{align}
  In particular, (\ref{positivity assumption on f}) translates into 
  \begin{align}\label{positivity assumption on varphi and I}
  \varphi >0 \; \; \text{ on } I=(0,\tau_2).
  \end{align}
  We can then express the metric $\omega$     obtained from Calabi's construction (\ref{spelling out calabis ansatz}) as  
  \begin{align} \label{spelling out calabis ansatz in terms of varphi}
  \omega
  = \pi^* \omega_M - \tau \pi ^* \gamma + \frac{1}{\varphi(\tau)} \sqrt{-1} \partial \tau \wedge \partialb \tau
  \end{align}
  by using equations (\ref{spelling out calabis ansatz}) and  (\ref{new relations after transformation}).  
  
  The function $\varphi$ is called the momentum profile of $\omega$. We note that it is possible to reconstruct the K\"ahler potential $f$ of $\omega$ from its momentum profile by
  \begin{align}\label{compute f from varphi}
  f(t)= \int_{0}^{\tau (t)} \frac{xdx}{\varphi(x)}.
  \end{align}
 Hence, the K\"ahler metric given by Calabi's Ansatz (\ref{spelling out calabis ansatz}) is uniquely determined by its momentum profile. We emphasize this by writing $\omega= \omega_{\varphi}$.

  \paragraph{Completeness of $\omega_\varphi$.}
  The K\"ahler metric $\omega= \omega_{\varphi}$ given by (\ref{spelling out calabis ansatz in terms of varphi}) is a priori only defined on $K_M\setminus M$ and is in general not complete. Whether or not $\omega_{\varphi}$ extends across the zero section to a complete metric is determined by the behaviour of the momentum profile $\varphi$ toward the endpoints of $I=(0,\tau_2)$. This is well-understood and there is the following well-known proposition, whose proof can be found in \cite{hwang2002momentum}[Section 2] or \cite{futaki2011constructing}[Section 6], for example.

  	\begin{prop}  \label{completness proposition}
  		Let  $\omega_\varphi$ be given by (\ref{spelling out calabis ansatz in terms of varphi}). 
  		Suppose the profile $\varphi: I \to \mathbb{R}$ has a zero of integer order at each endpoint of $I=(0,\tau_2)$. Then $\omega_\varphi$ extends across the zero section  if and only if $\varphi(0)=0$   and $\varphi' (0)=1$. 
  		
  		In this case, the resulting metric on $K_M$ is complete if and only if
  		at the upper endpoint $\tau_2$, one of the following conditions (i) and (ii) holds:
  			\begin{itemize}
  				\item [(i)] The endpoint $\tau_2$ is finite and $\varphi$ vanishes at least to second order.
  				\item [(ii)] The endpoint $\tau_2$ is infinite and $\varphi$ grows at most quadratically.
  				\end{itemize}

  	\end{prop}

\begin{remark}
	Note that  \cite{hwang2002momentum}[Proposition 2.3] is identical with Pro\-position \ref{completness proposition}, except that Hwang and Singer require $\varphi'(0)=2$ instead of $\varphi'(0)=1$. This  is due to the fact that our K\"ahler potential $f$ is \textit{twice} the potential function used by Hwang and Singer; compare (\ref{spelling out calabis ansatz}) with \cite{hwang2002momentum}[(1.1)].  
\end{remark}

If the metric $\omega_{\varphi}$ extends to the total space of $K_M$, we would like to identify its de Rham cohomology class. Since we assumed (\ref{positivity assumption on varphi and I}), i.e. $I=(0,\tau_2$), it follows immediately that $[\omega_\varphi] = [\pi ^* \omega_M] \in H^2(K_M)$. We refer to the class $[\pi^*\omega_M]$ as the K\"ahler class of $\omega_{\varphi}$. 

More generally, we define a K\"ahler class on $K_M$ simply to be a class in $H^2(K_M)$ containing positive $(1,1)$ forms and the K\"ahler cone is the set of all such K\"ahler classes. Using this definition, the projection map $   \pi^*: H^2(M) \to H^2(K_M) $ identifies the K\"ahler cone of the compact base $M$ with the K\"ahler cone of $K_M$. Indeed, given a K\"ahler form on $K_M$, its restriction to  $M$ clearly is a K\"ahler form on $M$. Conversely, given a K\"ahler form $\omega_M$ on $M$, Calabi's Ansatz always produces a positive $(1,1)$ form in the class $[\pi^*\omega_M]$, for example consider $\omega_\varphi$ with $\varphi(\tau)= \tau $, which extends to $K_M$ by Proposition \ref{completness proposition}.

  \paragraph{The Ricci form.} In this paragraph,
   we  provide a description of the Ricci-form of $\omega_\varphi$. The computations can be found, for example, in \cite{hwang2002momentum}[Section 2.1].
   
Denote the K\"ahler metric of $\omega_M$ by $g_M$ and the curvature form of $(K_M,h)$ by $\gamma$. It gives rise to an endomorphism $B: T^{1,0}M \to T^{1,0}M$ of the holomorphic tangent bundle, which is locally defined by $B:=g_M^{-1} \gamma= g_M^{ \bar{k} i} \gamma_{j\bar k}$. 
As in Theorem \ref{existence theorem line bundles}, we assume from now on that the eigenvalues of $B$ are constant over $M$. This condition is sufficient to reduce the soliton equation (\ref{soliton equation}) to an ODE.
These conditions  guarantee  that the function $Q:I\times M \to \mathbb{R}_{+}$ defined by 
\begin{align}\label{definition Q LB}
Q= \det \left( g_M^{-1} (\omega_M - \tau \gamma ) \right)= \det \left( \operatorname{Id} - \tau B
 \right) 
\end{align}
 only depends on the parameter $\tau$, i.e. is constant over $M$. Also observe that $Q$ is a positive function because $-\gamma\geq 0$ and $\tau\geq0$. $Q$ naturally appears in the computation of $\Ric (\omega_\varphi)$. Indeed,
 the Ricci form is given by
 \begin{align}\label{Ricci form spelled out}
 \Ric(\omega_\varphi) = \pi ^* \Ric(\omega_M) + \frac{(\varphi Q)^{'}}{Q}  \pi ^* \gamma - \frac{1}{\varphi} \left( \frac{(\varphi Q)^{'}}{Q} \right)^{'}  \sqrt{-1} \partial \tau \wedge \partialb \tau,
 \end{align}
 see \cite{hwang2002momentum}[(2.14)].

\subsection{Reduction to an ODE} We use the previously derived formula for the Ricci curvature to show that the K\"ahler-Ricci soliton equation is equivalent to an ODE in the function  $\varphi(\tau)$. Our presentation is similar to \cite{futaki2011constructing}[Section 4].

By definition, the soliton vector field $X$ must be the real part of a holomorphic vectorfield, i.e. $\mathcal{L}_XJ =0$. On the line bundle $K_M$, there is a natural choice for $X$, which we now describe.
 $K_M$ admits a holomorphic $\mathbb{C}^*$-action by fibre-wise multiplication and the corresponding holomorphic vector field $Z$ is given by $Z=z_0 \frac{\partial}{\partial z_0}$, where $z_0$ denotes the fibre coordinate of $K_M$. 
In terms of the radial function $t$ defined at the beginning of this section, we can write $Z$ as
\begin{align}\label{definition eulervectorfield}
Z =\operatorname{Re} Z + \sqrt{-1} \operatorname{Im} Z=      \partialt - \sqrt{-1}  J  \partialt.
\end{align}
So it is natural to  set $X:= \mu \operatorname{Re} Z = \mu \partialt$ for some constant $0\neq \mu\in \mathbb{R} $.
Before deriving the ODE, we need to calculate the following Lie-derivative:
\begin{align}\label{lie derivatie step 1}
	\mathcal{L}_X \omega_\varphi = d (\iota _X \omega_\varphi ) = \sqrt{-1} \partial \partialb (\mathcal{L}_X f )(t) = \mu \sqrt{-1} \partial \partialb  f'(t).
\end{align}
Here, we used  $2\sqrt{-1} \partial \partialb = d J d $ and $\mathcal{L}_XJ =0$ to obtain the second equality. We shall write out equation (\ref{lie derivatie step 1}) in terms of fibre and base direction, as we did for the Ricci-form in (\ref{Ricci form spelled out}):
\begin{align}\label{lie derivative final formula}
-	\mathcal{L}_X \omega_\varphi =  \mu
	\varphi (\tau)\pi^* \gamma - \mu \frac{\varphi'}{\varphi} (\tau) \sqrt{-1} \partial \tau \wedge \partialb \tau.
\end{align}
Now we are in position to see by comparing (\ref{Ricci form spelled out}) and (\ref{lie derivative final formula}) that the soliton equation (\ref{soliton equation}) for $\omega_\varphi$ is equivalent to the following two equations
\begin{align}
\Ric(\omega_M) +\frac{(\varphi Q)'}{Q} (\tau) \gamma  & = \mu \varphi (\tau ) \gamma \label{equation to differentiate} \\
 \left( \frac{(\varphi Q)'}{Q} \right)' (\tau)&=\mu \varphi'(\tau).   \label{equation to integrate}
\end{align}
Since $\operatorname{Ric}(\omega_M) = - \gamma $, we see that differentiating (\ref{equation to differentiate}) gives (\ref{equation to integrate}), so
that we proved the following Lemma:
\begin{lemma} \label{sumup lemma}Suppose that $\omega_{\varphi}$ is a K\"ahler metric  with momentum profile $\varphi$. Then  (\ref{soliton equation}) with $X= \mu \frac{\partial}{\partial t}$ is equivalent to the following equation:
	\begin{align}
	\varphi ' (\tau) + \left(\frac{Q'}{Q}(\tau) - \mu  \right) \varphi (\tau)&= 1 \label{ODE soliton}
	\end{align}
	
\end{lemma}
 For the rest of this paragraph, we study the solution $\varphi$ to Equation (\ref{ODE soliton}). This is  a linear ODE of the form $y' + p(x)y = q(x)$, which has an explicit one-parameter family of  solutions given by
 \begin{align}\label{general solution}
 y= \exp \left( - \int p(x) dx  \right) \left(   \int q(x) \exp \left( \int p(x) dx  \right) dx   +K \right).
 \end{align}
 Applying (\ref{general solution}) to (\ref{ODE soliton}), we have 
 \begin{align}\label{final solution of varphi}
 \varphi (\tau) = \frac{e^{\mu\tau} }{Q (\tau)} \left( \int_{0}^{\tau}  e^{-\mu x} Q (x) dx     + K \right),
 \end{align}
where $K\geq0$ is determined by the initial value $\lim_{\tau \to 0} \varphi(\tau)$. Justified by (ii) of Proposition \ref{completness proposition}, we will assume that $K=0$.

 One can compute the integral  (\ref{ODE soliton}) explicitly in terms of the coefficients $b_j \geq 0$  of the polynomial $
 Q (\tau)= \det \left(  \operatorname{Id} - \tau B\right)= b_k \tau^k + b_{k-1} \tau^{k-1}+ \cdots + b_0$. Note that the degree $k$ of $Q$ could be less than $n$ since $B$ is allowed to have zero eigenvalues. 
 In fact, it is straight forward to see that 
 \begin{align}\label{expression of varphi in terms of nu}
 	\varphi(\tau)= \nu (0) \frac{e^{\mu \tau}}{Q(\tau)}   - \frac{\nu(\tau)}{Q(\tau)}   ,
 \end{align}
 where $\nu$ is given by
  \begin{align}\label{expression of nu}
 \nu(\tau)=   \sum_{j=0}^{k}   \sum_{l=0}^{j} b_j \frac{j!}{l!}     \frac{\tau ^l }{\mu^{j+1-l} }   .
 \end{align}
 We point out that the explicit expression for $\nu$ is not relevant, but rather that it has the form
 \begin{align}\label{explicit solution for varphi}
 \varphi (\tau)= \nu(0) \frac{e^{\mu \tau}}{Q(\tau) } + \frac{  (-b_k/\mu )\tau^k+  R_{k-1} (\tau)}{Q(\tau)}
 \end{align}
 for  a polynomial  $R_{k-1}$ of degree $k-1$. 
Hence, we found an explicit solution for the soliton ODE (\ref{ODE soliton}). Also note that $\varphi$ is defined on $[0, +\infty)$ since $Q(0)>0$.  Moreover, $\varphi$ is clearly positive on $(0, +\infty)$.

With these observations, we can now finish the proof of Theorem \ref{existence theorem line bundles}.

\subsection{Proof of Theorem \ref{existence theorem line bundles}} \label{section proof of theorem LB}

Let $K_M \to (M,g_M)$ be the canonical  bundle whose semi-negative curvature form $\gamma=-\operatorname{Ric}(\omega_M) $ has constant eigenvalues w.r.to $g_M$. 
Suppose $\varphi: (0,+\infty) \to \mathbb{R}$ is given by (\ref{final solution of varphi}) with $K=0$ and, as before, let $\omega_{\varphi}$ be defined by (\ref{spelling out calabis ansatz in terms of varphi}). Since $\varphi(\tau)>0$  for all $\tau>0$, $\omega_{\varphi}$ defines a K\"ahler metric and hence is a steady K\"ahler-Ricci soliton by Lemma \ref{sumup lemma}. We note that these metrics can only be complete if $\mu<0$. This can be proven similarly to \cite{feldman2003rotationally}[Lemma 5.1].

Hence we assume $\mu<0$.  From (\ref{explicit solution for varphi}), we have  the following asymptotic  behaviour  for large $\tau$:
\begin{align} \label{asymptotics varphi steady}
\varphi (\tau) =  - \frac{1}{\mu } + O(1/ \tau). 
\end{align}
Also recall that $\varphi$ and the potential $f$ are related by
\begin{align}\label{ode for f from varphi}
\frac{df'}{dt} (t)= \varphi(f'(t)).
\end{align}
Using (\ref{ode for f from varphi}) together with (\ref{asymptotics varphi steady}), we conclude that the corresponding potential $f(t)$ is indeed defined for all $t \in \mathbb{R}$, i.e. $\omega_\varphi$ is defined on $K_M\setminus M$.

It remains to check that $\omega_{\varphi}$ extends across the zero section and defines a  complete metric as $t\to + \infty$.
By the first part of Proposition \ref{completness proposition}, $\omega_{\varphi}$ extends provided $\varphi(0)=0$ and $\varphi'(0)=1$. Since we assumed $K=0$ in (\ref{final solution of varphi}), we have $\varphi(0)=0$. Plugging this into (\ref{ODE soliton}) gives $\varphi'(0)=1$, as desired.
The completeness as $t\to +\infty$ follows immediately from the asymptotic expansion (\ref{asymptotics varphi steady}) and (ii) of Proposition \ref{completness proposition}.

\subsection{Examples.} \label{section LB examples}
 Theorem \ref{existence theorem line bundles} immediately recovers all known  examples  of steady K\"ahler-Ricci solitons on the total space of line bundles (\cite{Caosoliton}, \cite{chave1996class}, \cite{pedersen1999quasi}, \cite{dancer2011ricci}, \cite{yang2012characterization}). In these cases, the base is a product of K\"ahler-Einstein manifolds and the considered K\"ahler classes are represented by convex combinations  of  K\"ahler-Einstein metrics on each factor. 
 
 If the base manifold is a flag variety, Theorem \ref{existence theorem line bundles} produces examples, which have not been mentioned before. In this case, steady solitons sweep out the entire K\"ahler cone.

 \begin{example}[Products]\label{example steady line bundle}
 	 Let  $(M_i,\omega_i)$, $i=1,\dots, r $ be K\"ahler-Einstein manifolds with non-negative scalar curvature and denote their canonical bundles by $K_{M_i} \to M_i$. We consider the bundle 
 	\begin{align*}
 	K_M=p^*_1 K_{M_1} \otimes \cdots \otimes p^* _r K_{M_r} \to  M:=M_1\times \cdots \times M_r ,
 	\end{align*}
 	where $p_i:M \to M_i$ is the projection. Then Theorem \ref{existence theorem line bundles} applies and gives a complete steady soliton in each K\"ahler class of the form $\sum_{i=1}^r \alpha_i [p_i^*\omega_i]  \in H^2 (M)$ with $\alpha_i >0$. 

	 The case $r=1$ was first considered in \cite{Caosoliton} and \cite{chave1996class} for $M=\mathbb{CP}^n$ and in \cite{pedersen1999quasi} for a general K\"ahler-Einstein Fano manifold.
	 For $r>1$, these solitons are found in \cite{dancer2011ricci}[Theorem 4.20].
 \end{example}

\begin{example}[Flag varieties] \label{example flag varieties}
	Let $G$ be a complex semisimple Lie group, $P\subset G$ a parabolic subgroup and $K\subset G$ a maximal compact subgroup.  Then $K$ acts transitively on  the flag manifold $M= G/P$. It is well-known that $M$ admits a $K$-invariant complex structure so that its anti canonical bundle is ample, compare \cite{besse2007einstein}[Chapter 8] for example.
	
	The previously mentioned results only produce solitons on the canonical bundle $K_M$ whose K\"ahler class is a multiple of   $[\pi ^* c_1(M)] $. In general, however, $H^2(M)$ is not spanned by $[ c_1(M)]$. 
	
	We claim that \textit{every} K\"ahler class  admits a steady K\"ahler-Ricci soliton. Indeed, every K\"ahler class on $M$ admits a $K$-invariant K\"ahler form $\omega_K$ whose Ricci form $\operatorname{Ric}(\omega_K)$ is also $K$-invariant. This means that the eigenfunctions of $\operatorname{Ric}(\omega_K)$ w.r.t. $\omega_K$ must be $K$-invariant and hence constant  since $K$ acts transitively on $M$. So  Theorem \ref{existence theorem line bundles} can be applied and proves the existence of a steady soliton in the class [$\pi^*\omega_K$]. 
\end{example}

\begin{remark} \label{Remak I on Conlon and Deruelle}
	The new metrics in Example \ref{example flag varieties} can also be obtained from the recent result \cite{conlon2020steady}[Theorem A], which was posted after the first version of this paper was uploaded to the arXiv. 
\end{remark}

	\section{Calabi metrics on vector bundles} \label{section calabis ansatz for VB}

Given a vector bundle $E\to D$, Hwang-Singer's idea was to apply their construction to the tautological bundle $\mathcal{O}_{\mathbb{P}(E)}(-1)$ over $\mathbb{P}(E)$, the projectivization of $E$ (\cite{hwang2002momentum}[Section 3.2]). In this section, we explain the changes which are necessary to prove Theorem 1.2 and provide some examples.

The main difference is that one has to choose a new background metric on $\mathbb{P}(E)$, with respect to which the eigenvalues of the curvature form are computed. 
Then the discussion of the previous section can be applied and again, the soliton equation (\ref{soliton equation}) reduces to a simple ODE. In this new setting, however, the function $Q$ defined by (\ref{definition Q LB}) will have zeros at $\tau=0$, so there are some details which have to be checked.

\subsection{Constructing a K\"ahler metric} \label{Section VB metric construction}
As in \cite{hwang2002momentum}[Section 3.2], we explain how to adapt the machinery from the previous section to the tautological line bundle.

Let $\pi : E \to (D,\omega_D)$ be a holomorphic vector bundle of rank $m\geq 2$  equipped with an Hermitian metric $h$ and assume that the K\"ahler manifold $D$ has complex dimension $d$. As in the case of line bundles, we define $r: E \to \mathbb{R}_{\geq0}$ to be the radial function induced by $h$ and let $t= \log r^2$. Then Calabi's Ansatz has the form 
\begin{align*}
	\omega= \pi ^* \omega_D + \sqrt{-1} \partial \partialb f(t). 
\end{align*}

By construction, the projectivization of $E$ is naturally a fibre bundle $p:\mathbb{P}(E) \to D$, with fibre isomorphic to $\mathbb{CP}^{m-1}$. Recall that the natural map $ \OPE \subset p^* E  \to E$ identifies $\OPE \setminus \mathbb{P}(E) \cong E \setminus D$. By abuse of notation, we  denote the bundle projection of $\OPE$ also by $\pi$, so that we have a commuting diagram 
\begin{align}\label{diagram of tautological bundle}
\begin{CD}
	L:=\mathcal{O}_{\mathbb{P}(E)}(-1) @>>> E \\
	\pi @VVV  						 @VVV \pi  \\
	M:=\mathbb{P}(E) @>p>> 			D
\end{CD}	
\end{align}
In the notation from the previous section, let us denote the complex dimension of $M$ by $n$, i.e. $n= d+m-1$. Via the natural identification $L\setminus M \cong E\setminus D$, $h$ induces a Hermitian metric on $L$, which we also simply denote by $h$. Hence, we can view  $r$ as a function on $L$ and, if $\gamma$ is the curvature form of $(L,h)$,  we have as before $\pi ^ *\gamma=- \sqrt{-1} \partial \partialb t$ with $t=\log r^2$. 
We again assume that $-\gamma\geq0$.
Then we are looking for metrics of the form
\begin{align}\label{Calabi ansatz for VB}
\omega_{\varphi} =& \pi ^* \omega_D -f'(t) \pi^* \gamma + f''(t) \sqrt{-1} \partial t \wedge \partialb t 
\end{align}
where we  require that $f: \mathbb{R} \to \mathbb{R}$ satisfies (\ref{positivity assumption on f})
to obtain a positive form. As before, we set $\tau:= f'(t)$ and define $\varphi: (0,\tau_2)\to \mathbb{R}_{+}$ by (\ref{new relations after transformation}), so that it also satisfies (\ref{positivity assumption on varphi and I}). Hence, $\omega_{\varphi}$ can also be expressed as in (\ref{spelling out calabis ansatz in terms of varphi}). 

For the computation of  Ricci curvature below, we need to choose a background K\"ahler metric $\omega_M$ on $M$. Define
\begin{align}\label{reference metric on M}
\omega_M= p ^* \omega_D - \gamma, 
\end{align}
which is clearly positive in base direction of the fibration $p: \mathbb{P}(E) \to D$. To see that $\omega_M$ is positive in fibre direction, we note that $-\gamma$ restricts to the  Fubini-Study metric on each fibre $\cong \mathbb{CP}^{m-1}$.

\paragraph{The Ricci form.} The calculation is in principle the same as in the line bundle case, but the polynomial $Q$ does have zeros. Let $B= g_M ^{-1} \gamma $ be the curvature endomorphism of $\gamma$, where $g_M$ is the metric with K\"ahler form given by (\ref{reference metric on M}) and assume that the eigenvalues of $B$ are constant over $M$. Then we define a function $Q$ by 
\begin{align}\label{definition Q  VB}
Q = \det (g^{-1} _M (p^*\omega_D - \tau \gamma  )  ), 
\end{align}
 which can be viewed  as a function $Q: (0, \tau_2) \to \mathbb{R}_{\geq 0}$. Indeed, we can write
 \begin{align}\label{Q only depends on tau}
 g_M^{-1} \left(  p^*\omega_D - \tau \gamma \right)= g_M^{-1} \left(  \omega_M - (\tau-1)\gamma   \right) = \operatorname{Id} - (\tau-1)B,
 \end{align}
so that  $Q$ is constant over $M$, i.e. it only depends on $\tau$. If $\beta_1, \dots, \beta _n $ are the eigenvalues of $B$, we must have $\beta_{d+1}= \dots= \beta_n = -1$ by the definition of $\omega_M$ and $\beta_1, \dots, \beta_d \leq 0$ by assumption. From (\ref{Q only depends on tau}), we conclude that $Q$ is given by
\begin{align} \label{definition of Q hat}
Q(\tau) = \tau ^{n-d} \prod_{j=1}^{d} (1+ \beta_j - \tau \beta_j) = \tau^{n-d} \hat{Q} (\tau),
\end{align} 
for some polynomial $\hat{Q}$. Since $p^*\omega_D$ is positive in base direction, we conclude from  (\ref{Q only depends on tau}) that $1+\beta_j>0$ for all $j=1, \dots d$. Hence, $\hat{Q}(0)>0$ and  $Q$ has a zero at $\tau=0$ of order  $n-d=m-1$.

As in (\ref{Ricci form spelled out}), one can find the following expression for the Ricci form: 
\begin{align} \label{vectorbundle ricci form}
\operatorname{Ric}(\omega_\varphi) = \pi^* \operatorname{Ric}(\omega_M) + \frac{(\varphi Q)'}{Q}  \pi ^* \gamma - \frac{1}{\varphi} \left( \frac{(\varphi Q)'}{Q} \right) '  \partial \tau \wedge \partialb \tau .
\end{align}

\subsection{The ODE} The natural  $\mathbb{C}^*$-action on $E$ by biholomorphisms induces a holomorphic vector field $Z$.  On $L\setminus M$, which is the tautological bundle with the zero section removed, the real part of $Z$  is  given by $\operatorname{Re} Z=  \partial / \partial t$, so we  are looking for Ricci solitons with vector field $X= \mu  \partial / \partial t $. Again, we find 
\begin{align} \label{vector bundle lie derivative}
-\mathcal{L} _X \omega_{\varphi} =- \mu \sqrt{-1} \partial \partialb f'(t) = \mu \varphi (\tau) \pi ^* \gamma - \mu \frac{\varphi '}{\varphi} (\tau) \sqrt{-1} \partial \tau \wedge \partialb \tau.
\end{align}
Combining (\ref{Calabi ansatz for VB}) with (\ref{vectorbundle ricci form}) and (\ref{vector bundle lie derivative}), one can check that the soliton equation (\ref{soliton equation}) is equivalent to 
\begin{align}
\operatorname{Ric}(\omega_M) =c \gamma  \label{VB equation in base direction} \\
\varphi'(\tau ) + \left(  \frac{Q'}{Q}(\tau ) - \mu  \right) \varphi (\tau) =  - c \label{VB soliton ODE}
\end{align}
for some integration constant $c\in \mathbb{R}$. In fact, we must have $c=-m$ since the first Chern class  of $M=\mathbb{P}(E)$ is given by 
\begin{align} \label{canonical bundle of P(E)}
c_1(M)= - m c_1(\OPE) + p^*c_1(E) + p^*c_1(D).
\end{align}
Equation (\ref{VB soliton ODE}) has the same form as (\ref{ODE soliton}), but with a different $Q$. Hence,  the solution $\varphi$ is given by
 \begin{align}\label{VB final solution of varphi}
\varphi (\tau) = \frac{e^{\mu\tau} }{Q (\tau)} \left( \int_{0}^{\tau} m e^{-\mu x} Q (x) dx    \right),
\end{align}
if we assume the integration constant to be zero. 

We end this section by studying the solution $\varphi$.
Let us write $Q(\tau)= b_{k+n-d} \tau^{k+ n-d} + \cdots + b_{n-d} \tau^{n-d} $ with coefficients $b_j\geq0$ for $j=1+ n-d, \dots, k+n-d$ and $b_{n-d}=\hat{Q}(0)>0$. Adapting  (\ref{expression of varphi in terms of nu}) and (\ref{expression of nu}) to this case, we obtain
\begin{align}\label{VB varphi in terms of nu}
		\varphi(\tau)= \nu (0)  \frac{e^{\mu \tau}}{Q(\tau)}  - \frac{\nu(\tau)}{Q(\tau)}
\end{align}
as well as 
\begin{align}\label{VB definition of nu}
\nu(\tau)= m  \sum_{j=n-d}^{k+n-d}    \sum_{l=0}^{j}  b_j \frac{j!}{l!}   \frac{\tau ^l }{\mu^{j+1-l} }   .
\end{align}
A priori, $\varphi$ given by (\ref{VB final solution of varphi}) is defined on the interval $(0,+\infty)$ and because $Q(0)=0$ one needs to check that $\varphi$ and its derivatives have a limit as $\tau\to 0$. To see that this is the case, note that we can rewrite (\ref{VB definition of nu}) as
\begin{align*}
\nu(0) e^{\mu \tau} 	-\nu(\tau) = m \sum_{j=n-d}^{k+n-d} b_j \frac{ j!}{\mu ^{j+1}} \sum_{\l=j+1}^{\infty} \frac{(\mu \tau)^l}{l!},
\end{align*}
i.e. $\tau^{-(n-d)}(\nu(0) e^{\mu \tau} 	-\nu(\tau))$ tends to zero as $\tau \to 0$. Since $Q$ vanishes of order $n-d$ at $\tau =0$, we then deduce from (\ref{VB varphi in terms of nu}) that $\lim_{\tau \to 0} \varphi=0$. Similarly, it follows that all derivatives of $\varphi$ have a limit as $\tau \to 0$.

\subsection{Proof of Theorem \ref{existence theorem VB}} \label{section proof of theorem VB}

The proof is now analogue to Section \ref{section proof of theorem LB}.
The only part that might a priori be different is the extension of $\omega_{\varphi}$ to a complete metric on $E$. However, one can check that
 Proposition \ref{completness proposition} also applies to the vector bundle case, see \cite{hwang2002momentum}[Lemma 3.7].

As before,  one can check that $\varphi$ has the behaviour required by Proposition \ref{completness proposition}. Indeed, one can compute that  $\varphi(0)=0$ and $\varphi' (0)=1$, as desired. Sending $\tau \to + \infty$, we conclude the following asymptotic expansion 
from (\ref{VB final solution of varphi}) and (\ref{VB varphi in terms of nu})
\begin{align}\label{asymptotics varphi VB}
\varphi(\tau) = -\frac{m}{\mu } +O(1/\tau),
\end{align}
and so we  obtain a complete metric on the total space $E$.

\subsection{Examples}

We briefly discuss  three  different situations to which Theorem \ref{existence theorem VB} applies. New examples of steady solitons are given in Example \ref{example sum of line bundles}.

\begin{example}[Complex plane] We let $D$ be a single point and $E\cong \mathbb{C}^n$ be the trivial bundle over $D$. Let $h$ be the Euclidean metric on $E$, so that $\omega_M=-\gamma$ is the Fubini-Study metric on $M= \mathbb{CP}^{m-1}$. This is the situation first studied in \cite{Caosoliton}. 

\end{example}
	
\begin{example}[Sum of line bundles] \label{example sum of line bundles}
Let $(D,\omega_D)$ be a K\"ahler-Einstein Fano manifold of Fano index $m$. Define $L:= K_D^{1/m}$ and consider the $m$-fold sum of $L$ with itself, i.e. $E= L \otimes \mathbb{C}^m$. Then we have $M=\mathbb{P}(E)= \mathbb{CP}^{m-1} \times D$ and 
\begin{align*}
\mathcal{O}_{\mathbb{P(E)}}(-1) = p_1^* \mathcal{O}_{\mathbb{CP} ^{m-1}} (-1)\otimes p_2^* L,
\end{align*} 
where $p_1,p_2$ denote the projections onto the first and second factor of $M$, respectively. Let $\omega_{FS}$ be the Fubini-Study metric on $\mathbb{CP}^{m-1}$, so that $\gamma= -p_1^*\omega_{FS} - 1/m p_2^*\Ric(\omega_D) $ is the curvature form of $\mathcal{O}_{\mathbb{P(E)}}(-1)$,  and define $\omega_M= p_2^* \omega_D - \gamma$. Then we clearly have 
\begin{align}
\operatorname{Ric}(\omega_M)= m p_1 ^* \omega_{FS} + p_2^*\Ric(\omega_D) = -m\gamma,
\end{align}
since $\omega_D$ is K\"ahler-Einstein. Moreover, the eigenvalues of $\Ric (\omega_M)$ w.r.t. $\omega_M$ are constant, so that Theorem \ref{existence theorem VB} can be applied. These examples of steady solitons are obtained in \cite{li2010rotationally}[Theorem 2.1] and \cite{dancer2011ricci}[Theorem 4.20]. 

If the base $D=G/P$ is a flag manifold for $G$ a complex semisimple Lie group and $P\subset G$ a parabolic subgroup, one can find steady solitons in \textit{every} K\"ahler class, similarly as in Example \ref{example flag varieties}. 

To see this, assume that $\omega_D$ represents a given K\"ahler class (not necessarily the first Chern class of $D$). We can pick $\omega_D$ to be $K$-invariant, where $K\subset G$ is a maximal compact subgroup. Since $\Ric (\omega_D)$ is also $K$-invariant, the form $-\gamma= p_1^* \omega_{FS} + 1/m p_2^* \operatorname{Ric}(\omega_D)$ is invariant under the diagonal action of $SU(m) \times K$ and also positive. 

We claim that $\operatorname{Ric}(-\gamma)= - m \gamma$. By \cite{besse2007einstein}[Theorem 8.2], we know that there exists a $SU(m)\times K$-invariant K\"ahler-Einstein metric $\omega_{\text{KE}} \in c_1(\mathbb{CP}^{m-1} \times D)$. Also recall that  the Ricci forms of \textit{all} $SU(m)\times K$-invariant K\"ahler metrics agree, i.e. $\operatorname{Ric}(-\gamma)= \operatorname{Ric}(\omega_{\text{KE}})$. Since $-m\gamma$ and $\omega_{\text{KE}}$ are in the same K\"ahler class, we deduce from the uniqueness part of Calabi's conjecture that $-m \gamma = \omega_{\text{KE}}= \operatorname{Ric}(-\gamma)$. 

As the form $\omega_M= p_2^ * \omega_D -\gamma$ is also invariant under $SU(m) \times K$, we conclude
\begin{align*}
\operatorname{Ric}(\omega_M)= \operatorname{Ric}(-\gamma) = -m \gamma,
\end{align*}
and hence the assumptions in Theorem \ref{existence theorem VB} are satisfied.
\end{example}

\begin{example}[Cotangent bundle of $\mathbb{CP}^d$]

Let $D= \mathbb{CP}^d $ be projective space equipped with the Fubini-Study metric  and consider $E= T^* \mathbb{CP}^d$, the cotangent bundle of $\mathbb{CP}^d$. $E$ is naturally a $SU(d+1)$-homogeneous vector bundle, where the fibre action is given by the coadjoint action of $SU(d+1)$ on its Lie algebra. 
Since $\mathbb{CP}^d$ is a rank 1 symmetric space, the induced action of $SU(d+1)$ on $M=\mathbb{P}(E)$ is transitive. Verifying the assumptions of Theorem \ref{existence theorem VB} is now similar to the previous Example.
These  steady solitons  on $T^*\mathbb{CP}^d$ are of cohomogeneity one and are contained in  \cite{dancer2011ricci}[Section 5].

\end{example}
	  
\begin{remark}\label{Remak II on Conlon and Deruelle}
	All the previous examples can also be constructed  from \cite{conlon2020steady}[Theorem A]. 
\end{remark}

	\section{Properties of $\omega_{\varphi}$ }

In this short section, we study curvature properties of the previously constructed metric $\omega_{\varphi}$. We show that $\omega_{\varphi}$ has bounded curvature and that its Ricci curvature is non-negative.   Moreover, we obtain  estimates on the  growth of the function $f$ and its derivatives.  

Recall that $f=f(t)$ is the K\"ahler potential of $\omega_{\varphi}$ as defined in (\ref{spelling out calabis ansatz}) and $\varphi=\varphi(\tau)$ is its momentum profile, see (\ref{spelling out calabis ansatz in terms of varphi}).  If $\omega_{\varphi}$ is a steady K\"ahler-Ricci soliton constructed in Theorem \ref{existence theorem line bundles} or Theorem \ref{existence theorem VB}, then $\varphi$ satisfies (\ref{ODE soliton}) or (\ref{VB soliton ODE}), respectively. This ODE is in turn determined by the polynomial $Q=Q(\tau)$ defined by either (\ref{definition Q LB})  or (\ref{definition Q  VB}). The statements in this section mainly reduce to understanding $Q$ and how it effects the asymptotic behaviour of $\varphi$, compare (\ref{expression of varphi in terms of nu}) or (\ref{VB varphi in terms of nu}) depending on the rank of the underlying vector bundle.

 We begin by considering  the Ricci curvature of $\omega_{\varphi}$. More precisely, we prove the following theorem, which we need in the subsequent section. It generalises the observation made in \cite{yang2012characterization}[Case 7].

\begin{theorem}\label{sign of ricci curvature Thm}
The complete steady K\"ahler-Ricci solitons constructed in Theorem \ref{existence theorem line bundles} and \ref{existence theorem VB} have non-negative Ricci curvature. Moreover, if the curvature form $-\gamma$ is positive definite, then the Ricci curvature is positive away from the zero section.
\end{theorem}

\begin{proof}
	First, we consider the solitons constructed on line bundles in Theorem \ref{existence theorem line bundles}.
	Let $\omega_{\varphi}$ be the K\"ahler metric given by (\ref{spelling out calabis ansatz in terms of varphi}) with $\varphi$ satisfying (\ref{ODE soliton}) and $\varphi(0)=0$.  Recall that the Ricci curvature is given by
	\begin{align*}
	\Ric(\omega_{\varphi}) = - \mathcal{L} _X (\omega_{\varphi})= \mu \varphi (\tau) \pi^* \gamma - \mu \frac{\varphi'}{\varphi} \sqrt{-1} \partial \tau \wedge \partialb \tau.
	\end{align*}
	Since  $\varphi(0)=0$, $\varphi>0$ on $(0,\infty)$, and $\mu \gamma \geq 0$, we only need to show that $\varphi'> 0$. To see that this is the case, we define a function 
	\begin{align}
	H(\tau):= 
	\frac{Q^2}{Q' - \mu Q} e^{-\mu \tau} - \int_{0}^ \tau  e^{-\mu x}Q(x)dx .
	\end{align}
	Using the ODE (\ref{ODE soliton}), it is straight forward to prove that $\varphi'\geq 0$ iff $H\geq 0$. As $H(0)>0$ for $Q$ given by (\ref{definition Q LB}), we are done if we can show that $H'\geq0$. From the definition of $H$, we compute
	\begin{align}
	H'(\tau)=  e^{-\mu\tau} \frac{Q}{(Q'-\mu Q)^2} \left( (Q')^2 - Q Q''   \right),
	\end{align}
	so that $H'\geq0$ if and only if $(Q')^2 - Q Q'' \geq0$. The later condition can be checked easily starting from the explicit expression for $Q$. Indeed, let $\beta_1,\dots,\beta_n$ be the eigenvalues of the endomorphism $B=g_M ^{-1} \gamma : T^{1,0}M \to T^{1,0}M$, and write 
	\begin{align}
	Q(\tau) = \det\left( \operatorname{Id}- \tau B\right)=\prod_{j=1}^{n}(1-\beta_j \tau).
	\end{align}
	Then we have
	\begin{align}\label{Q' ^2 - Q  Q'' }
	\frac{(Q')^2- Q Q''}{Q^2} = \sum_{j=1}^n \frac{\beta_j ^2}{(1-\beta_j \tau)^2} \geq0, 
	\end{align}
	as required. For the second statement, it suffices to observe that $\varphi'(\tau)>0$ if and only if $(Q')^2 - Q Q'' >0$, which is certainly true   if $\gamma<0$. This proves Theorem \ref{sign of ricci curvature Thm} for line bundles.
	
	The arguments for the metrics in Theorem \ref{existence theorem VB} are analogous. It also reduces to showing that $(Q')^2- QQ'' \geq 0$, where $Q$ is this time given by (\ref{definition of Q hat}). 	
\end{proof}

Note that the non-negativity of Ricci curvature can also be expressed in terms of the potential function $f$. In particular, we have the following
\begin{corollary}\label{corollary f'''>0}
	Let $\omega_{\varphi}$ be a steady K\"ahler-Ricci soliton constructed in Theorem \ref{existence theorem line bundles} or \ref{existence theorem VB} and let $f=f(t)$ be defined by (\ref{spelling out calabis ansatz}) or (\ref{Calabi ansatz for VB}), respectively. Then $f''$ is monotone increasing.
\end{corollary}

\begin{proof}
	Recall from (\ref{lie derivatie step 1}) or (\ref{vector bundle lie derivative}) that we have 
	\begin{align*}
		 - \mathcal{L}_X \omega_{\varphi}=- \mu \sqrt{-1} \partial \partialb f'(t) =  \mu f''(t) \pi ^* \gamma -\mu f'''(t) \sqrt{-1} \partial t \wedge \partialb t, 
	\end{align*}
	and so $\operatorname{Ric}(\omega_{\varphi})= - \mathcal{L}_X \omega_{\varphi} $ can only be non-negative if $f'''(t)\geq 0$ since $\mu<0$. Thus, Theorem \ref{sign of ricci curvature Thm} implies that $f''$ is increasing. 
\end{proof}
We end this section by pointing out some growth properties of the potential function $f$.

\begin{lemma}\label{growth of derivatives of f}
	Let $\omega_{\varphi}$ be a steady K\"ahler-Ricci soliton constructed in Theorem \ref{existence theorem line bundles} or \ref{existence theorem VB} and let $f=f(t)$ be related to $\varphi$ by (\ref{compute f from varphi}). Then   there is a constant $C>0$ such that for all $t\geq C$, we have
	\begin{align}\label{f' grows like the function t}
	C^{-1} \leq f''(t) \leq C \;\; \text{ and } \;\; C^{-1} t\leq  f'(t) \leq C t.
	\end{align}
	Moreover, for all $j\in \mathbb{N}_0$ and $t\geq C$ 
	\begin{align}\label{in Lemma growth of derivatives of f}
	C^{-1} (1+f'(t))^{-j} \leq |	f^{(2+j)}(t)|\leq C(1+f'(t))^{-j}.
	\end{align}
	
\end{lemma}

\begin{proof}
	 First note that the bound on $f''(t)$ in (\ref{f' grows like the function t})  implies the bound on $f'(t)$ after integrating the parameter $t$, so we only need to find $C>0$ such that 
	\begin{align}\label{in Lemma bounding derivatives of f': claim for f''}
		C^{-1} \leq f''(t) \leq C
	\end{align}
	for all $t\geq C$. Translating the problem into bounding  $\varphi(\tau)$, we recall from (\ref{new relations after transformation}) that 
	\begin{align} \label{in Lemma bounding derivatives of f': relations tau and f}
		\tau=\tau(t)=f'(t)   \;\; \text{ and } \;\; \varphi(\tau(t))=f''(t).
	\end{align}
	Since $f'(t)$ is positive and increasing, we can choose a $C\geq 1$ such that   the following estimate 
	\begin{align}\label{in Lemma bounding derivatives of f': first lower bound on tau and f'}
		\tau(t) = f'(t) \geq C^{-1}
	\end{align}
	holds for all $t\geq C$. Then we  recall the asymptotic expansion (\ref{asymptotics varphi VB})
	\begin{align*}
		\varphi(\tau(t)) = -\frac{m}{\mu} + O(1/\tau(t))
	\end{align*}
	with $\mu <0$ implying that $\varphi( \tau(t))$ is uniformly bounded from above because of (\ref{in Lemma bounding derivatives of f': first lower bound on tau and f'}). Together with (\ref{in Lemma bounding derivatives of f': relations tau and f}), this proves the upper bound for $f''(t)$ in (\ref{in Lemma bounding derivatives of f': claim for f''}). 
	For the lower bound, note that $f''(t)>0$ is increasing and thus is bounded from below by some positive constant if $t\geq C$. Inequality (\ref{in Lemma bounding derivatives of f': claim for f''}) now follows, and so does (\ref{f' grows like the function t}).

	Next, consider the case $j>0$, i.e. we estimate $f^{(2+j)}(t)$. Differentiating (\ref{in Lemma bounding derivatives of f': relations tau and f}) and using the chain rule, we see that
	\begin{align*}
		f'''(t) = \varphi'(\tau(t)) \frac{d \tau}{dt}(t)= \varphi'(\tau(t))\cdot f''(t).
	\end{align*}
	Taking further derivatives of this equation, we conclude that $f^{(2+j)}$ can be written as
	\begin{align}\label{in Lemma growht of derivatives of f'': sum formula for explicit expression}
		f^{(2+j)}= \sum_{\alpha} c_\alpha \cdot\varphi^{(\alpha_1)} \cdot \ldots \cdot \varphi^{(\alpha_i)} \cdot (f'')^{\,j},
	\end{align}
	where the sum is over all multi-indices $\alpha$ with $\alpha_1+\ldots + \alpha_i =j$ and $c_\alpha$ are constants only depending on the multi-index $\alpha$. Since $f''(t)$ satisfies (\ref{in Lemma bounding derivatives of f': claim for f''}), it is sufficient to estimate derivatives of $\varphi$. In fact, we have for all $\beta \in \mathbb{N}$ that
	\begin{align}\label{in Lemma growth of derivatives of f'': asympotic behaviour of derivatives of varphi}
	C^{-1 } \tau^{-\beta}\leq |\varphi^{(\beta)}(\tau) |\leq C \tau^{-\beta},
	\end{align}
	because  $\varphi(\tau)$ behaves asymptotically like a rational function  of the form $P/Q$ with polynomials $P(\tau),Q(\tau)$  having the same degree, see (\ref{explicit solution for varphi}). Substituting $\tau(t) = f'(t)$ in (\ref{in Lemma growth of derivatives of f'': asympotic behaviour of derivatives of varphi}) and combining the resulting estimate with (\ref{in Lemma growht of derivatives of f'': sum formula for explicit expression}), we finally obtain (\ref{in Lemma growth of derivatives of f}) as desired. 
\end{proof}

The important point about Lemma \ref{growth of derivatives of f} is estimate (\ref{f' grows like the function t}), i.e. that $f''(t)$ behaves like a constant and $f'(t)$ growths roughly like the function $t$ in the limit $t\to \infty$. 
 This will be crucial in the next section because we want to understand the asymptotic geometry of $\omega_{\varphi}$.

Another interesting consequence of Lemma \ref{growth of derivatives of f} is that the  metrics $\omega_{\varphi}$ have  bounded curvature and positive injectivity radius.
\begin{lemma} \label{curvature of g varphi is bounded}
	The curvature tensor of the steady solitons constructed in Theorem \ref{existence theorem line bundles} and \ref{existence theorem VB} is uniformly bounded and each of these metrics has positive  injectivity radius. 
\end{lemma}

\begin{proof}
	It is straight forward to see that the first claim reduces to bounding $f''(t), f''' (t)$ and $f^{(4)}(t)$, where $\varphi$ and $f$ are related by (\ref{compute f from varphi}), so we focus on the second one. 
	
	According to \cite{cheeger1982finite}[Theorem 4.7], the lower bound on the injectivity radius follows if we can bound the volume of all unit balls uniformly from below. For this, recall that the function $t$ identifies $E\setminus D \cong \mathbb{R} \times S$, where $S$ is the $S^1$-bundle associated to $\mathcal{O}_{\mathbb{P}(E)}(-1) \to \mathbb{P}(E)$, see (\ref{diagram of tautological bundle}). 
	Under this identification, the metric $g_\varphi$ admits the following decomposition on $\mathbb{R} \times S$
	\begin{align}\label{in Lemma on injectivity radius: metric g varphi in coordinates}
	g_\varphi = f''(t) \left( dt^2  + \left(Jdt\right)^2  \right) + f'(t) \pi ^*\hat{g} + \pi ^*g_D,
	\end{align}
	where $J$ denotes the complex structure on $E$, and $\hat{g}$, $g_D$ are the (2,0) tensors associated to $-\gamma$, $\omega_D$, respectively. By compactness of $D$, we only have to consider the set $\{t\gg 1\}$, on which  $g_\varphi$ is uniformly equivalent to the metric 
	\begin{align}\label{metric uniformly equivalent to g varphi}
		g_t:=dt^2  + \left(Jdt\right)^2 + t \pi ^*\hat{g} + \pi ^*g_D,
	\end{align}
	compare Lemma \ref{growth of derivatives of f}. Let us further denote $g_{S^1}:= (Jdt)^2$ and rescale $g_t$ by some fixed constant so that the diameter $\operatorname{diam}(S,g_{S^1})$ of each $S^1$-fibre  satisfies $\operatorname{diam}(S,g_{S^1})=1/4$. It then suffices  to bound the volume of unit balls w.r.t. $g_t$ on the set $\{t\gg 1\}$ uniformly from below. 

	Let $x\in E$ with $t(x)\gg 1$ and denote the  unit ball of $g_t$ around $x$ by $B_{g_t} (x,1)$. We introduce families of metrics on $M$ and $S$ by declaring 
	\begin{align*}
		g_{M,\tau} &:= \tau \hat g + p^* g_D \\
		g_{S,\tau}&:= g_{S^1} + \pi ^* g_{M,\tau}
	\end{align*}
	for each $\tau \geq 1$, where $p:M \to D$ is the projection as in (\ref{diagram of tautological bundle}) and $\pi : S\to M$. In particular, the projection $\pi $ becomes a Riemannian submersion $\pi :(S,g_{S,\tau}) \to (M,g_{M,\tau})$. Using this notation and writing $x=(t(x),y)$, we obtain the following inclusion
	\begin{align*}
		B:= [t(x)-1/2,t(x)] \times B_{g_{S,t(x)}} (y, 1/2) \subset B_{g_t}(x,1).
	\end{align*}
	This is an immediate consequence of the decomposition (\ref{metric uniformly equivalent to g varphi}) together with the fact that for all $p \in B$ we  have $t(p)\leq t(x)$ and $g_{S,\tau_0} \leq g_{S,\tau_1}$ for all $\tau_0\leq \tau _1$.
	Before estimating  the $g_t$-volume $\operatorname{Vol}_{g_t}(B_{g_t}(x,1))$ of the unit ball $B_{g_t}(x,1)$, we observe that 
	\begin{align}\label{g S t -1/2 from below}
		g_{S,t(x)-1/2} = g_{S,t(x)} - \frac{1}{2} \pi ^* \hat g \geq g_{S,t(x)} - \frac{1}{2} (t(x)-1) \pi ^* \hat g \geq \frac{1}{2} g_{S,t(x)}
	\end{align}
	provided $t(x)\geq 2$. Using the inclusion $B\subset B_{g_t}(x,1)$ then implies that
	\begin{align*}
		\operatorname{Vol}_{g_t}(B_{g_t}(x,1))
		 &\geq \operatorname{Vol}_{g_t} (B )\\
		&\geq  \frac{1}{2} \cdot \operatorname{Vol}_{g_{S,t(x) -1/2}} ( B_{g_{S,t(x)}} (y,1/2) )\\
		&\geq 2^{-\frac{\operatorname{dim}_{\mathbb{R}}S }{2} -1} \operatorname{Vol}_{g_{S,t(x)}} ( B_{g_{S,t(x)}} (y,1/2) ),
	\end{align*}
	where we applied Fubini's theorem in the second line, and the last inequality follows from (\ref{g S t -1/2 from below}). Thus, it remains to bound the $g_{S,t(x)}$-volume of $ B_{g_{S,t(x)}} (y,1/2) $ uniformly from below. 
	
	We further reduce this volume bound to an integration on $M$ by observing that the projection $\pi : S \to M$ satisfies
	\begin{align}\label{inclusion of balls along pi}
		\pi ^{-1} (B_{g_{M,t(x)}} (\pi (y), 1/4 )  ) \subset B_{g_{S,t(x)}} (y, 1/2)
	\end{align}
	Indeed, given a $b \in B_{g_{M,t(x)}} (\pi (y), 1/4 )$ and a length-minimizing curve $q:[0,1] \to M$ from $\pi (y)$ to $b$, we may lift $q$ to a horizontal curve $\tilde q$  in $S$ from $\tilde{q}(0)=y$ to some point $\tilde{q}(1) \in \pi^{-1}(b)$. For any $a \in \pi ^{-1}(b)$,  the triangle inequality for the distance function $\operatorname{dist}_{g_{S,t(x)}}$ then yields
	\begin{align*}
		\operatorname{dist}_{g_{S,t(x)}} (y,a) 
		&\leq \operatorname{dist}_{g_{S,t(x)}} (y, \tilde{q}(1)) +\operatorname{dist}_{g_{S,t(x)}} (\tilde{q}(1), a)  \\
		&\leq  \operatorname{dist}_{g_{M,t(x)}} (\pi (y), b)   +\frac{1}{4} \\
		&< \frac{1}{4} + \frac{1}{4} = \frac{1}{2},
	\end{align*}
	where the second inequality holds since we normalised each fibre $\pi ^{-1}(b)$ to be of diameter $1/4$ and the third one follows since $\pi :S \to M$ is a Riemannian submersion. Hence, we conclude that $ a \in B_{g_{S,t(x)} (y,1/2)}$ as claimed.  
	
	Inclusion (\ref{inclusion of balls along pi}) yields an estimate on the $g_{S,t(x)}$-volume as follows. We write $\omega_{t(x)}$ for the K\"ahler form of $g_{M,t(x)}$ and $\chi$ for the characteristic function of the ball $B_{g_{S,t(x)}}(y,1/2)$, and then observe that
	\begin{align*}
		\int _{B_{g_{S,t(x)}} (y,1/2)  } (Jdt) \wedge \pi ^* \omega^{\operatorname{dim}_{\mathbb{C} M}    } _{t(x)} &= \int _{B_{g_{M,t(x)}} (\pi(y),1/2  )  } \pi _* (\chi  Jdt) \cdot \omega_{t(x)} ^{\operatorname{dim}_\mathbb{C} M}  \\
		& \geq  \int _{B_{g_{M,t(x)}} (\pi(y),1/4  )  } \pi _* ( Jdt) \cdot \omega_{t(x)} ^{\operatorname{dim}_\mathbb{C} M}  \\
		& = \operatorname{Vol}_{g_{S^1}} (S^1) \cdot \int _{B_{g_{M,t(x)}} (\pi(y),1/2  )  }  \omega_{t(x)} ^{\operatorname{dim}_\mathbb{C} M}.  
	\end{align*} 
	Here, $\pi_* (\chi Jdt) $ denotes the function on $M$ obtained by integrating $\chi Jdt$ over fibres, i.e. $\pi _* (\chi Jdt) (b) = \int_{\pi^{-1}(b)} \chi Jdt$, so that the first equality follows from Fubini's theorem. In the second line, we used that $\chi \equiv 1$ on the set $\pi ^{-1} (B_{g_{M,t(x)}}   (\pi (y)  ,1/4))$ by (\ref{inclusion of balls along pi}) and the final equation holds because  the volume of each $S^1$-fibre is the same by (\ref{metric uniformly equivalent to g varphi}).  Thus, it remains to 
		find a constant $C>0$, independent of $x=(t(x),y)$, such that 
	\begin{align}\label{the desired lower volume bound for g M t(x)}
		\operatorname{Vol}_{g_{M,t(x)}} (B_{g_{M,t(x)}} (\pi(y) ,1/4)) \geq C^{-1} >0.
	\end{align}

To prove this, let us first assume that $\hat g$ is positive definite. Then we can find a constant $C_0>0$, only depending on the eigenvalues of $\hat g$ w.r.t. $g_M$, such that 
\begin{align*}
	C_0^{-1} \tau g_M \leq g_{M,\tau} \leq C_0 \tau g_M,
\end{align*}
for all $\tau \geq 1$ and with $g_M:= g_{M,1}$. Since we can rescale by a fixed constant, it suffices to bound the $\tau g_M$-volume of $B_{\tau g_M}(z,1) $ from below by a constant independent of both $z \in M$ and $\tau \gg 1$. We note that
\begin{align*}
	B_{\tau g_M} (z,1) = B_{g_M}(z,\tau ^{-\frac{1}{2}})
\end{align*}
and by compactness, the $g_M$-volume $\operatorname{Vol}_{g_M}(B_{g_M} (z,\tau ^{-\frac{1}{2}}) )$ is, up to some uniform constant, bounded from below by $\tau^{-\frac{\operatorname{dim}_{\mathbb{R}M}}{2}}$ for $\tau $ sufficiently large. This shows that $\operatorname{Vol}_{\tau g_M} (B_{\tau g_{M}}   (z,1)) = \tau^{\frac{\operatorname{dim}_{\mathbb{R}M}}{2}} \operatorname{Vol}_{g_M}(B_{g_M} (z,\tau ^{-\frac{1}{2}}) )$ is indeed uniformly bounded from below and (\ref{the desired lower volume bound for g M t(x)}) then follows.

Let us now assume that $\gamma$ has at least one zero eigenvalue w.r.t. $g_M$. Since these eigenvalues are assumed to be constant over $M$, its Kernel $\operatorname{Ker}\gamma$ defines a proper subbundle of $T^{1,0}M$. Moreover, $\gamma$ is  closed, so that $\operatorname{Ker}\gamma$ is integrable according to Frobenius' theorem.
 Thus, if $n$ is the complex dimension of $M$ and $k$ the number of positive eigenvalues of $\gamma$, we find  a chart around each point defined on some neighborhood of the Euclidean unit ball $B(1)\subset \mathbb{C}^n\cong \mathbb{C}^k \times \mathbb{C}^{n-k}$ around the origin such that each slice $\{z_0\} \times \mathbb{C}^{n-k}$ in $B(1)$ is an integral manifold for $\operatorname{Ker}\gamma$, i.e. 
 \begin{align*}
 	\gamma(v,v)>0 \;\; \text{ and } \;\; \gamma(v,w)=0 \;\; \text{ for all } v\in T\mathbb{C}^k, \, w \in T\mathbb{C}^{n-k}.
 \end{align*}
 By compactness, we can cover $M$ by finitely many of such Euclidean balls $B_j(1)$ for $j=1,\dots,N$ and also find a constant $C_0>0$ such that
 \begin{align*}
 	C_0^{-1} g_{\mathbb{C}^n} \leq g_M \leq C_0 g_{\mathbb{C}^n} \;\; \text{ on each } B_j(1),
 \end{align*}
 where $g_{\mathbb{C}^n}$ is the Euclidean metric on $B_j(1)$ and the constant $C_0$ is independent of the ball $B_j(1)$.
 
  For $\tau\geq 1$, we also   consider the following product metric 
 \begin{align*}
 	g_{\mathbb{C}^n ,\tau} := (1+\tau ) g_{\mathbb{C}^k} + g_{\mathbb{C}^{n-k}} \;\; \text{ on } \;\; \mathbb C ^n \cong \mathbb{C}^k \times \mathbb C^{n-k}
 \end{align*}
 Then there exists a uniform constant $C>0$, which only depends on $C_0$ and the $g_M$-eigenvalues of $\gamma$, such that 
 \begin{align}\label{g C tau equivalent to g M tau}
 	C^{-1} g_{\mathbb{C}^n,\tau} \leq g_{M,\tau} \leq C g_{\mathbb{C}^n,\tau} \;\; \text{ on each } \;\; B_j(1).
 \end{align}
Let $\varepsilon>0$ be the Lebesgue number associated to the cover $\{B_j(1)\}_{j=1,\dots N}$ of the manifold $(M,g_M)$, i.e.  the ball $B_{g_M}(z,\varepsilon)$ is contained in $B_j(1)$ for some $j$.
Note that since $g_M\leq g_{M,\tau}$, we also have $B_{g_{M,\tau}} (z,1) \subset B_{g_M}(z,1) \subset B_j(1)$. Additionally, we may assume that $\varepsilon<1$, so that it suffices to bound the $g_{M,\tau}$-volume of the smaller ball $B_{g_{M,\tau}} (z,\varepsilon)$  from below because the constant $\varepsilon>0$ is independent of both $z\in M$ and $\tau \geq 1$. 

This, in turn, can be reduced to bounding the $g_{\mathbb{C}^n,\tau}$ -volume of $B_{g_{\mathbb{C}^n,\tau}} (z,C^{-\frac{1}{2}} \varepsilon)$.  Indeed, this is a direct consequence of  (\ref{g C tau equivalent to g M tau}), which implies that the volume forms of $g_{\mathbb{C}^n,\tau}$ and $g_{M,\tau}$ are uniformly equivalent on $B_{g_{M,\tau}}(z,\varepsilon)$ and also that we have the inclusion
\begin{align*}
	B_{g_{\mathbb{C}^n,\tau}} (z, C^{-\frac{1}{2}} \varepsilon ) \subset B_{g_{M,\tau}}(z,\varepsilon). 
\end{align*} 
For the remaining lower volume bound, observe that the following product of Euclidean balls
\begin{align}\label{product of euclidean balls for fubini}
	B_{g_{\mathbb{C}^k}} \left(z, \frac{1 }{2} C^{-\frac{1}{2}} \varepsilon(1+\tau)^{-\frac{1}{2}}\right) \times B_{g_{\mathbb{C}^{n-k}}} \left(z, \frac{1 }{2} C^{-\frac{1}{2}} \varepsilon\right)
\end{align}
is contained in $B_{g_{\mathbb{C}^n,\tau}} (z,C^{-\frac{1}{2}}  \varepsilon )$. Applying Fubinis' theorem to the product (\ref{product of euclidean balls for fubini}) and using the fact that the volume form of $g_{\mathbb{C}^n,\tau}$ is equal to $(1+\tau)^k$-times the volume form of $g_{\mathbb{C}^n}$ then yields the required lower bound on the $g_{\mathbb{C}^n,\tau}$-volume of $B_{g_{\mathbb{C}^n,\tau}} (z, C^{-\frac{1}{2} } \varepsilon)$, which is independent of both $z\in M$ and $\tau \geq 1$. This finishes the proof.

\end{proof}

	\section{Uniqueness in a K\"ahler class} \label{section uniqueness}

The purpose of this section is to prove Theorem \ref{uniqueness theorem in intro}. We begin by briefly recalling notation from Sections \ref{section calabis ansatz for line bundles} and \ref{section calabis ansatz for VB} and then define the function spaces appearing in Theorem \ref{uniqueness theorem in intro}. We also explain how to reduce the proof to a $\partial \partialb$-Lemma, which is stated below (Theorem \ref{del del bar lemma}).

\subsection{A $\partial \bar \partial$-Lemma}
Throughout this section, let $\pi : E\to D$ be a rank $m$ holomorphic vector bundle over a compact K\"ahler manifold $(D,\omega_D)$. The complex dimension of $E$ (as a manifold) is denoted by $m+ d$, where $d$ is the complex dimension of $D$. If $m=1$, we assume that it satisfies the conditions in Theorem \ref{existence theorem line bundles}, and if $m\geq 2$, we assume $E$ is given as in Theorem \ref{existence theorem VB}. Also recall that we defined a radial function $r: E\to \mathbb{R}_{\geq0} $ by $r(v)= \sqrt{h(v,\bar v)}$, which vanishes along the zero section of $E$ and we set $t:= 2\log r$. Note that we can use the function $t$ to identify $E$, with its zero section removed, as the product $\mathbb{R} \times S$, where $S$ is the $S^1$-bundle associated to $\mathcal{O}_{\mathbb{P}(E)}(-1)\to \mathbb{P}(E)$, see Diagram (\ref{diagram of tautological bundle}).  Under this identification, the function $t$ on $E\setminus D$ corresponds to the projection onto the first factor of $\mathbb{R} \times S$. 

Let $\omega_{\varphi}$ be the K\"ahler Ricci soliton constructed in Theorem \ref{existence theorem line bundles} or \ref{existence theorem VB}, i.e. $\omega_{\varphi} $ is defined by (\ref{spelling out calabis ansatz in terms of varphi}) with $\varphi$ satisfying (\ref{ODE soliton}) if $E$ is a line bundle or by (\ref{Calabi ansatz for VB}) and (\ref{VB soliton ODE}) if $m \geq 2$. We denote the corresponding Riemannian metric by $g_\varphi$.

On the manifold $\mathbb{R} \times S$, we can write the metric $g_\varphi$ as follows. If  $J$ denotes  the complex structure on $E$ and  $g_D$ and $\hat{g}$ are the (2,0) tensors associated to $\omega_D$ and $-\gamma$, respectively, then
\begin{align}\label{metric g varphi in coordinates}
g_\varphi = f''(t) \left( dt^2  + \left(Jdt\right)^2  \right) + f'(t) \pi ^*\hat{g} + \pi ^*g_D,
\end{align}
where $f$ can be reconstructed from $\varphi$ via (\ref{compute f from varphi}). 
We would also like to point out that we allowed $-\gamma$ to have zero-eigenvalues, i.e. $\hat{g}$ is only semidefinite. As a consequence, the volume growth of $g_\varphi$ will be determined by the  zero-eigenvalues of $-\gamma$. 

Before stating the main theorem of this section, 
we require a definition of  weighted function spaces. As a weight function, we choose $w: E \to \mathbb{R}_+ $ to be defined by
\begin{align}\label{definition of weight function}
w(t):= 1+ f'(t).
\end{align}
This choice is inspired by the work of Hein \cite{hein2011weighted}. Indeed, the following lemma shows that $w$ has the same properties as the function $\rho$ in \cite{hein2011weighted}[Theorem 1.6]. 

 \begin{lemma} \label{weight function equivalent to distance}
	Fix $x_0\in E$   and denote the distance function of $g_\varphi$ by $\rho(x)$. Then there exists a constant $C>0$ such that 
	\begin{align}\label{in lemma: w is comparable to the distance function}
	C^{-1}w(t(x))\leq (1+\rho(x)) \leq C w(t(x))
	\end{align}
	for all $x\in E$ with $w(t(x))\geq C$. Moreover, $w$ satisfies 
	\begin{align} \label{in lemma: w has bounded gradient and Lapalce bound}
	|\nabla w| + w |\Delta w| \leq C,
	\end{align}
	where $|\cdot|, \nabla$ and $\Delta$ are associated with $g_\varphi$. 
\end{lemma}

\begin{proof}
	We identify $E\setminus D \cong \mathbb{R} \times S$ and without loss of generality, we can assume $x_0=(t_0,y_0) \in \mathbb{R}\times S$. Let $(t,y) \in \mathbb{R}\times S$ with $t_0\leq t$ and consider a shortest path $q_{t,y}=(q_t,q_y):[0,1] \to \mathbb{R} \times S$ from $(t_0,y_0)$ to $(t,y)$. Its length $L(q_{t,y}) $ is given by 
	\begin{align*}
	L(q_{t,y}) = \int_0^1 \sqrt{g_\varphi \left( \dot{q}_{t,y} (\sigma), \dot{q}_{t,y} (\sigma) \right)} d\sigma.
	\end{align*}
	Then (\ref{in lemma: w is comparable to the distance function}) reduces to finding a constant $C>0$ such that 
	\begin{align}\label{in lemma: required bound of length function}
	C^{-1}	w(t) \leq L(q_{t,y}) \leq C w(t)
	\end{align}
	for all $y\in S$ and all $t\geq C$.  In fact, it is sufficient to show inequality (\ref{in lemma: required bound of length function}) with $w(t)$ replaced by $t$ since there is a $C>0$ such that 
	\begin{align} \label{in lemma: w equivalent tu t}
	C^{-1} t \leq w(t) \leq Ct
	\end{align}
	for $t\geq C$, compare (\ref{f' grows like the function t}).
	Thus, we begin by choosing $C>0$ such that (\ref{in lemma: w equivalent tu t}) holds, and we increase $C>0$ as we go along, if necessary. 
	
	For proving the lower bound in (\ref{in lemma: required bound of length function}), we estimate
	\begin{align*}
	L(q_{t,y})\geq \int _0^1 \sqrt{f''(q_{t,y})} \dot{q_t}(\sigma) d \sigma \geq \sqrt{f''(t_0)} (t-t_0),
	\end{align*}
	as required. 
	Before showing the upper bound, we conclude from (\ref{f' grows like the function t}) that
	\begin{align*}
	g_\varphi \leq C \left( dt^2 + t g_S   \right),
	\end{align*}
	where we define $g_S:= (Jdt)^2 + \pi ^* \hat{g} + \pi ^* g_D$ with $\hat{g}$ and $g_D$ as in (\ref{metric g varphi in coordinates}). Also observe that we can now assume $q_t$ to be the linear path in the $\mathbb{R}$-factor, i.e. $q_t(\sigma) = \sigma (t-t_0)+ t_0$. Then we obtain
	\begin{align} \label{in Lemma: equivalence of dist, upper bound on distance}
	L(q_{t,y}) &\leq C \int_0^1 \dot{q}_t (\sigma) d\sigma + C\int_0 ^1 \sqrt{q_t (\sigma) } \cdot\sqrt{g_S(\dot{q_y}(\sigma), \dot{q_y}(\sigma)      } d\sigma \nonumber \\ 
	&\leq C t+ C\operatorname{diam}(S,g_S) \sqrt{t}  \\
	&\leq C t, \nonumber
	\end{align}
	for all $t$ sufficiently large and with $\operatorname{diam}(S,g_S)$ denoting the diameter of the compact manifold $(S,g_S)$. Now (\ref{weight function equivalent to distance}) follows immediately. 
	
	For the second claim, observe from (\ref{Calabi ansatz for VB}) that we have
	\begin{align*}
	|\nabla w|=  f'',
	\end{align*}
	which is uniformly bounded according to Lemma \ref{growth of derivatives of f}. For bounding the Laplace operator $\Delta w= \Delta f'$, recall that on a K\"ahler manifold, the Laplace operator $\Delta $ satisfies $\Delta = 2 \operatorname{tr}_{\omega_\varphi} \sqrt{-1} \partial \partialb $, where $\operatorname{tr}_{\omega_\varphi}$ denotes the trace computed w.r.to $g_\varphi$. Then we apply (\ref{vector bundle lie derivative}) to obtain
	\begin{align*}
	\Delta f'& = 2 \operatorname{tr}_{\omega_\varphi} \left( \sqrt{-1} \partial \partialb f' \right)\\
	&= \frac{2}{\mu} \operatorname{tr}_{\omega_\varphi}  \left( \mathcal{L}_X\omega_{\varphi}\right) \\
	&= 2 \left( \varphi \frac{Q'}{Q}+ \varphi'\right),
	\end{align*}
	where the last equality holds since there is the following formula
	\begin{align*}
	\operatorname{tr}_{\omega_\varphi}(-\pi ^* \gamma) = \frac{Q'}{Q},
	\end{align*}
	see \cite{hwang2002momentum}[(2.22)]. Using the soliton ODE (\ref{VB soliton ODE}), we continue
	\begin{align*}
	w \Delta w  &= (1+f') \Delta f'\\
	& = 2 (1+\tau) \left(  \varphi \frac{Q'}{Q} + \varphi'\right) \\
	&= 2 (1+\tau)\left(m+ \mu \varphi  \right),
	\end{align*}
	which is also uniformly bounded because of the asymptotic expansion (\ref{asymptotics varphi VB}). This, together with the uniform bound on $|\nabla w|$, implies (\ref{in lemma: w has bounded gradient and Lapalce bound}).
\end{proof}

Lemma \ref{weight function equivalent to distance} ensures that our definition of weighted function spaces below coincides with the one used in \cite{hein2011weighted}. These spaces are well-adapted to study the Laplace operator on a wide class of complete manifolds. 

\begin{defi}
	Let $\Lambda ^*T^* E$ be the exterior algebra of $T^*E$ and consider $\delta\in \mathbb{R}$ and $k\in \mathbb{N}_0$. We define $C^k_\delta (\Lambda ^* T^*E)$ to be the space of $k$-times continuously differentiable sections $\eta$  of $\Lambda ^* T^*E$ such that the norm
	\begin{align*}
	||\eta ||_{C^k_\delta} := \sum_{j=0}^k \sup _E |w^{j-\delta} \nabla ^j \eta  |
	\end{align*}
	is finite, where $w$ is given by (\ref{definition of weight function}) and $\nabla$, $|\cdot|$ are associated to $g_\varphi$. We also set
	\begin{align*}
	C^{\infty}_\delta (\Lambda ^* T^*E) := \bigcap_{k\in \mathbb{N}_0} C^k _\delta(\Lambda ^* T^*E).
	\end{align*}
\end{defi}
 In other words,  elements in $C^{\infty}_{\delta}(\Lambda^* T^*E) $ grow at most like $w^{\delta}$ and their $l$-th derivatives at most like $w^{\delta -l}$. 
 Having introduced the necessary notation, we can now state the main result of this section.

\begin{theorem}\label{uniqueness theorem}
	Let $\omega_{\varphi}$ be a steady K\"ahler-Ricci soliton constructed in Theorem \ref{existence theorem line bundles} or \ref{existence theorem VB}. Assume that $\omega$ is a K\"ahler-Ricci soliton on $E$ with the same vector field as $\omega_{\varphi}$ such that  $[\omega]=[\omega_{\varphi}] \in H^2 (E)$. If moreover $\omega_{\varphi}-\omega \in C_{-\delta}^{\infty}(\Lambda^2 T^*E)$ for some $\delta>2$, then $\omega_{\varphi}= \omega$. 
	
\end{theorem}

The main part of proving Theorem \ref{uniqueness theorem} will be 
a $\partial \partialb$-Lemma, with controlled growth. 
In fact, we will prove 

\begin{theorem} \label{del del bar lemma}
	Let $\delta>2$  and $\eta \in C^{\infty}_{-\delta} (\Lambda^*T^* E) $ be a real $(1,1)$ form. If $\eta$ is d-exact, then $\eta= \sqrt{-1} \partial \partialb u$ for some $u\in C^{\infty}_{2-\delta}(E)$. 
\end{theorem}
Assuming this result,  Theorem \ref{uniqueness theorem} follows immediately.
\begin{proof}[Proof of Theorem \ref{uniqueness theorem}]
	By Theorem \ref{del del bar lemma}, there exists a $u\in C^{\infty }_{2-\delta}(E)$ such that $\omega_{\varphi}- \omega = \sqrt{-1} \partial \partialb u$. Since $2- \delta<0$, $u$ and all its derivatives tend to zero at infinity, so we can apply the maximum principle \cite{Biquard17}[Proposition 1.2] and  conclude that $\omega_{\varphi}= \omega$.
\end{proof} 

The remainder of this section is devoted to proving Theorem \ref{del del bar lemma}. 
 We follow the ideas for asymptotically conical metrics given in  \cite{HAJOricciflat}[Section 3], which rely on two main ingredients. Firstly, we need to understand solutions to Poisson's equation $\Delta u =h$ and their growth behaviour (Section \ref{section laplace operator}). 
 Secondly, we need to show that harmonic (1,0) forms of certain growth behaviour are identically zero (Section \ref{section vanishing theorem}). 
 The proof of Theorem \ref{del del bar lemma} will then be finished in Section \ref{del del bar lemma}.

\subsection{The Laplace Operator.} \label{section laplace operator} We start by considering the Laplace operator 
$\Delta$ of the metric $g_\varphi$ acting on suitably weighted H\"older spaces, which we now define.

\begin{defi}
	Let $\operatorname{dist}(x,y)$ be the distance between $x,y\in E$ measured w.r.to $g_\varphi$ and denote the injectivity radius of $g_\varphi$ by $i_0$. (Note that $i_0>0$ by Lemma \ref{curvature of g varphi is bounded}). For  $0<\alpha<1$ and $\delta \in \mathbb{R}$, we define a seminorm on the space of all tensor fields $T$ on $E$ by 
	\begin{align*}
		[T]_{C^{0,\alpha} _{\delta}}:= \sup_{\substack{x\neq y \in E\\ \operatorname{dist}(x,y)<\frac{i_0}{2} } }\left( \min (w(x),w(y) )^{-\delta} \frac{|T_x - T_y|}{\operatorname{dist}(x,y)^\alpha}   \right),
	\end{align*}  
	where the norm $|\cdot|$ is induced by $g_\varphi$ and the difference $T_x-T_y$ is defined by using  parallel transport along the minimal geodesic from $x$ to $y$. 
	
	The weighted H\"older space $C^{k,\alpha}_\delta(E)$ is then defined to be the subset of all $u\in C^{k}_{\delta}(E)$ for which the norm
	\begin{align*}
	||u||_{C^{k,\alpha}_\delta} := ||u||_{C^{k}_\delta} + [\nabla^k u ]_{C^{0,\alpha}_{\delta-k-\alpha}}
	\end{align*}
	is finite. 
\end{defi}

The Laplace operator $\Delta$ acts as
\[
\Delta: C^{2,\alpha} _{2+\delta } (E) \to C^{0,\alpha}_{\delta }(E), 
\]
 for any $\delta \in \mathbb{R}$ and we are interested in the surjectivity of this operator. A partial answer to this question is provided in \cite{hein2011weighted}.
 
  Given $h \in C^{0,\alpha}_\delta(E)$ with $\delta<-2$, we can essentially always solve Poisson's equation $\Delta u=h$, but it is not clear how the solution $u$ will behave as $t\to \infty$. This  depends on the volume growth  of $g_\varphi$, which is related to the degree $k$ of the  polynomial $Q$ defined in (\ref{definition Q LB}) for $m=1$ or (\ref{definition of Q hat}) for $m\geq 2$. Alternatively, it is evident from the definition of $Q$ that $k$ is equal to $m+d-1$ minus the number of zero-eigenvalues of $\gamma$. (Recall that $m+d-1$ is the complex dimension of $\mathbb P (E)$.)

More precisely, we will prove the following important proposition about the existence of solutions to $\Delta u =h$. 
\begin{prop}\label{Laplace is isomorphism}
	Let $\delta>2$ and suppose $h\in C^{0,\alpha}_{-\delta }(E)$.
	\begin{itemize}
		\item [(i)] If $k\leq 1$,  assume $\int h \omega_{\varphi}^{m+d} =0$ additionally. Then there exists a $u \in C^{2, \alpha}(E)$ such that $\Delta u=f$ and the integral $\int |\nabla u|^2 \omega_{\varphi}^{m+d}$ is finite.
		
		\item [(ii)] If $k>1$ and $2<\delta<k+1$, then there exists $u \in C^{2,\alpha}(E) $ such that $\Delta u=h$ and $u= O(w^{2-\delta +\varepsilon}) $ as well as $ |\nabla u|=O(w^{2-\delta+\varepsilon})$ for all $\varepsilon>0$.
	\end{itemize}
\end{prop}
Before proceeding with its proof, we first of all need to check that we can indeed apply Hein's work \cite{hein2011weighted}[Theorem 1.5, 1.6], i.e. we have to verify that the metric $(E,g_\varphi)$ satisfies Hein's condition SOB($\beta$). For the sake of completeness, we recall \cite{hein2011weighted}[Definition 1.1] here.
 \begin{defi}[{\cite{hein2011weighted}[Definition 1.1]}]  \label{definition SOB}
 	A Riemannian manifold 
 	$(M,g)$ is called SOB($\beta$) if there exists a $x_0 \in M$ and a constant $C\geq 1$ satisfying the following:
 	\begin{itemize}
 		\item[(i)] The set $B(x_0,s_1) \setminus \overline{B}(x_0,s_0)$ is connected for all $s_1>s_0\geq C$, 
 		\item[(ii)] $\operatorname{Vol} (B(x_0,s)) \leq C s^{\beta}$ holds for all $s\geq C$,
 		\item[(iii)] $\operatorname{Vol} (B(x,(1-C^{-1})\rho (x))) \geq C^{-1} \rho(x)^{\beta}$ holds for all $x\in M$ with $\rho(x)\geq C$,
 		\item[(iv)] $\operatorname{Ric}_x \geq -C \rho(x)^{-2}$ holds if $\rho(x)\geq C$.  
 	\end{itemize}
 Here $B(x_0,s)$ denotes the geodesic ball around $x_0$, $\operatorname{Vol}(B(x_0,s))$ its volume and $\rho(x)$ denotes the distance from $x$ to $x_0$. 
 \end{defi}
As the next lemma shows, the soliton metrics $(E,g_\varphi)$ constructed in Theorem \ref{existence theorem line bundles}  and \ref{existence theorem VB} are SOB($k+1$).

\begin{lemma} \label{lemma check SOB}
	The metric $(E,g_\varphi)$ is SOB($k+1$), where $k$ is equal to $m+d-1$ minus the number of zero-eigenvalues of the curvature form $\gamma$ on $\mathbb P (E)$. 
\end{lemma}
\begin{proof}
	We fix $x_0\in D \subset E$ to be a point on the zero-section of $E$. Thanks to Theorem \ref{sign of ricci curvature Thm},  Condition (iv) in Definition \ref{definition SOB} is clearly satisfied, so we focus on (i), (ii), (iii). 
	
	Beginning with the volume estimates (ii) and (iii),   we consider the volume form of $g_\varphi$, which is given by
	\begin{align}\label{volume form for growth behaviour}
	\frac{\omega_{\varphi} ^{m+d}}{(m+d)!} = \frac{\sqrt{-1}}{(m+d)!} f''Q(f') \partial t\wedge \partialb t \wedge \left( \pi^* \omega_D  - \pi ^* \gamma \right)^{m+d-1},
	\end{align}
	where the polynomial $Q$ is defined by (\ref{definition Q LB}) if $m=1$ or (\ref{definition Q  VB}) if $m\geq 2$. 
	Recall from above, that the degree of $Q$ is equal to $k$ as defined in Lemma \ref{lemma check SOB}. If we then choose
 $C\geq 1$ such that  (\ref{f' grows like the function t}) is satisfied, we obtain for large $t\geq C$:
	\begin{align}\label{in Lemma SOB: growth of Q in terms of t}
		C^{-1} t^k\leq f''(t)Q(f'(t))\leq Ct^k . 
	\end{align}
	Moreover, Lemma \ref{weight function equivalent to distance} implies that 
	\begin{align}\label{in Lemma SOB: t comparible to rho}
		C^{-1} t(x) \leq \rho(x) \leq Ct(x)
	\end{align}
	 if $\rho(x)\geq C$. In the estimates that follow, we increase $C>0$ if necessary but it still denotes a uniform constant which only depends on the geometry of $(E,g_\varphi)$ and the choice of base point $x_0$.
	
	For verifying (ii), let $s\geq C$ and observe that (\ref{in Lemma SOB: t comparible to rho}) implies
	\begin{align*}
		B(x_0,s)\subset B(x_0,C) \cup \left\{ y \in E \, | \,  0 \leq t(y) \leq Cs   \right\}.
	\end{align*}
	Integrating over these sets and using (\ref{volume form for growth behaviour}), we obtain
	\begin{align*}
		\operatorname{Vol}(B(x_0,s)) &\leq C + C \int _0 ^{Cs} \int_S f''(t) Q(f'(t)) dt \wedge \partialb t \wedge \left(\pi ^*(\omega_D -\gamma)   \right)^{m+d-1} \\
		&\leq C + C\int_0 ^{Cs} t^k dt\\
		&\leq C s^{k+1},
	\end{align*}
	where we used (\ref{in Lemma SOB: growth of Q in terms of t}) in the second line. This proves (ii) of Definition \ref{definition SOB} with $\beta=k+1$. 
	
	For showing (iii), the goal is to choose a new $C_0\geq 1$ such that for all $x\in E$ with $\rho(x)\gg 1$ sufficient large, we have an inclusion of the form
	\begin{align}\label{in lemma SOB: inclusion of volume for lower bound}
		B(x,(1-C_0^{-1})\rho(x)) \supset \left\{   t(y) \in \left[ t(x)+1, t(x) + C_0^{-1} \rho(x) - \sqrt{\rho(x)}\right]\right\}.
	\end{align}
	Indeed, if (\ref{in lemma SOB: inclusion of volume for lower bound}) holds, we can integrate and use (\ref{in Lemma SOB: growth of Q in terms of t}) to estimate
	\begin{align*}
		\operatorname{Vol}(B(x,(1-C_0^{-1}) \rho(x)   )) &\geq C_0^{-1} \int_{t(x)+1} ^{t(x)+ C_0^{-1}\rho(x) - \sqrt{\rho(x)}  } \sigma ^k d\sigma\\
		&\geq C_0^{-1} \rho(x)^{k+1},
	\end{align*}
	which is (iii) with $\beta=k+1$ as required. Hence it remains to check inclusion (\ref{in lemma SOB: inclusion of volume for lower bound}). To see that this is true, we again introduce the metric $g_S:= (Jdt)^2 +\pi^* \hat{g} + \pi^* g_D$ on the cross-section $S$ as in the proof of Lemma \ref{weight function equivalent to distance}, so that 
	\begin{align}\label{in Lemma SOB: definition of g_t}
		g_\varphi \leq C\left( dt^2 + t g_S  \right)=: g_t.
	\end{align}
	To estimate the distance function of $g_t$ from above, we proceed as in (\ref{in Lemma: equivalence of dist, upper bound on distance}). Given  $x,y \in E$ with $C\leq t(x)$ and $t(x) \leq t(y)$, we consider a path $q:[0,1] \to E$ from $q(0)=x$ to $q(1)=y$, which we write as $q=(q_1,q_2)$ under the identification $E\setminus D \cong \mathbb{R} \times S$. Furthermore, we assume that $q_1(\sigma)= \sigma(t(y)-t(x)) + t(x)$ is the linear path from $t(x)$ to $t(y)$, so that we estimate using (\ref{f' grows like the function t})
	\begin{align*}
	\operatorname{dist}_{g_t} (x,y)& \leq C \int_0^1\dot{q}_1(\sigma) d\sigma + C\int_0^1 \sqrt{q_1(\sigma)} \sqrt{g_S(\dot{q}_2(\sigma),\dot{q}_2(\sigma))} d\sigma \\
	&\leq C (t(y)-t(x)) + C \operatorname{diam}(S,g_S) \left(\sqrt{t(y)-t(x)} + \sqrt{t(x)} \right)
	\end{align*}
	 Together with (\ref{in Lemma SOB: t comparible to rho}) and (\ref{in Lemma SOB: definition of g_t}), this implies
	\begin{align}\label{in Lemma SOB: dist between x,y in terms of tx,ty }
		\operatorname{dist}_{g_\varphi} (x,y) \leq C \left(t(y)-t(x) +\sqrt{t(y)-t(x)}  \right) +  C\sqrt{\rho(x)}
	\end{align}
	from which we can deduce inclusion (\ref{in lemma SOB: inclusion of volume for lower bound}). Indeed, let $C>0$ satisfy (\ref{in Lemma SOB: dist between x,y in terms of tx,ty }) and define a new constant $C_0>0$ by $C_0^{-1}= C^{-1}(1-C^{-1})$. If we then assume that $\rho(x)\gg 1$ is large enough so that $C_0^{-1} \rho(x) -\sqrt{\rho(x)} >1$, we estimate for all $y\in E$ with $ t(x)+1\leq  t(y) \leq t(x) +C_0^{-1} \rho(x) - \sqrt{\rho(x)} $:
	\begin{align*}
		\operatorname{dist}_{g_\varphi} (x,y)&\leq C \left(t(y)-t(x)  \right) +  C\sqrt{\rho(x)} \\
		&\leq C C_0^{-1} \rho(x) -C\sqrt{\rho(x)} +C\sqrt{\rho(x)}\\
		&=\left(1-C^{-1}\right) \rho(x).
	\end{align*}
	Here we obtained the first inequality by applying $t(y)-t(x)\geq 1$ to  (\ref{in Lemma SOB: dist between x,y in terms of tx,ty }) and the second inequality makes use of the upper bound on $t(y)$. This shows inclusion (\ref{in lemma SOB: inclusion of volume for lower bound}) and thus (iii). 
	
	It remains to verify Condition (i). By compactness of $D$, we can choose $C>1$ such that for all $s\geq C$ the ball $B(x_0,s)$ contains a tubular neighborhood of the zero section. Given $x\in E\setminus D$, we denote the complex line thorough $x$ by $L_x\cong \mathbb{C}$. We need to understand the shape of the intersection of $L_x$ with the set $B_{(s_0,s_1)}(x_0):= B(x_0,s_1) \setminus \overline{B}(x_0,s_0)$ for all $s_1>s_0\geq C$. 
	
	First, we claim that for each $x\in B(x_0,s)$, the radial path $q_{\operatorname{rad}}$ in $L_x$ from $x$ to $0 \in L_x$ is entirely contained in the ball $B(x_0,s)$. Note that for this to be true it suffices to show that the function $\rho$ is increasing along $q_{\operatorname{rad}}$. In order to prove this, use the identification $E\setminus D \cong \mathbb{R} \times S$ and write $x=(a_1,b)$. Let $q:[0,1] \to E$ be a shortest geodesic from $q(0)=x_0$ to $q(1)= (a_1,b)$. On  $E\setminus D$, we decompose $q=(q_1,q_2)$ and let us assume for the moment that $ q_1(\sigma)$ is increasing in $\sigma \in [0,1]$. 
	 Given  $a_0<a_1$, we then choose a $\sigma_0 \in (0,1)$ with $q_1(\sigma_0)= a_0$ and reparameterize the path $q$ by declaring $q_{\sigma_0} (\sigma) := (q_1(\sigma_0  \sigma), q_2(\sigma))$, so that $q_{\sigma_0}$ is a path from $x_0$ to $(a_0,b)$. It follows from (\ref{metric g varphi in coordinates}) that  we have
	\begin{align*}
		g_\varphi\left( \dot{q}_{\sigma_0}(\sigma), \dot{q}_{\sigma_0}(\sigma) \right) \leq 		g_\varphi\left( \dot{q}(\sigma), \dot{q}(\sigma) \right)
	\end{align*}  
	for all $\sigma\in [0,1]$, since $f''$ and $f'$ are both increasing and we assumed that $q_1(\sigma_0 \sigma) \leq q_1(\sigma)$. 
	Then  we conclude $L(q_{\sigma_0}) \leq L(q)$ and thus $\rho(a_0,b)\leq \rho(a_1,b)$ for all $a_0<a_1$ and $b\in S$, as we claimed.
	
	 Hence, the claim holds if we show that $q_1$ is increasing.
	Recall that by definition, $q_1= t(q)$ and clearly $q_1$ increases if and only if $r^2(q)= e^{t(q)}$ does, where $r: E\to \mathbb{R}_{\geq 0}$ is defined at the beginning of Section \ref{section uniqueness}. Since $x_0 $ lies on the zero section of $E$, we have $r(q(0))=0$ and consequently there is a $\hat \sigma \in [0,1)$ such that 
	\begin{align*}
	  r(q(\sigma)) =0\; \text{ for all }  \sigma \in[0,\hat{\sigma}]\;\; \text{ and } \;\; \frac{\text{d}}{\text{d}\sigma} r^2(q( \sigma))>0 \;\; \text{ on }\; (\hat \sigma, \hat  \sigma + \varepsilon)
	\end{align*}
	for some small $\varepsilon>0$. 
	In particular,  $\lim_{\sigma \to {\hat \sigma} ^+} \dot q_1(\sigma) \geq 0$ and we only have to rule out the existence of two points $\sigma_1, \sigma_2\in [\hat \sigma,1]$ with $ \sigma_1<\sigma_2$ such that
	\begin{align}\label{definition sigma 1 sigma 2}
		q_1(\sigma_1) = q_2(\sigma_2) \;\;\; \text{ and } q_1(\sigma_1) <q_1(\sigma) \;\;\text{ for all } \sigma \in (\sigma_1,\sigma_2).
	\end{align} 
	However, if this was the case, then the path $q$ cannot be length-minimizing. Indeed, suppose that there are such numbers $\sigma_1,\sigma_2$ satisfying (\ref{definition sigma 1 sigma 2}). Then we define a new path $\tilde q$ from $x_0$ to $x$ by 
	\begin{align*}
		\tilde{q} (\sigma) =
		\begin{cases}
			q (\sigma) & \text{if } \sigma \in [0,1] \setminus (\sigma_1,\sigma_2) \\
			(q_1(\sigma_1), q_2(\sigma)) & \text{if } \sigma \in (\sigma_1,\sigma_2).
		\end{cases}
	\end{align*} 
	Using the decomposition  (\ref{metric g varphi in coordinates}) and the fact that $f''$ is increasing, we see that 
	\begin{align*}
		L(\tilde q) <L(q),
	\end{align*}
	 contradicting  the minimality of $q$. It follows that $q_1$ must be increasing.

	Now we can verify Condition (i), so consider any $s_1>s_0\geq C$. As shown in the previous paragraph,  both $L_x\cap B(x_0,s_j)$ with $j=1,2$ are star-shaped regions with center $0\in L_x$, so the complement $L_x\cap B_{(s_0,s_1)} (x_0)$ is diffeomorphic to a genuine open annulus in $\mathbb{C}$. From this, we deduce that $B_{(s_0,s_1)}(x_0)$ is a fibre bundle over $\mathbb{P}(E)$ with annuli in $\mathbb{C}$ as fibres. In particular, $B_{(s_0,s_1)}(x_0)$ is connected because $D$ is, finishing the proof. 
\end{proof}

Before proving Proposition \ref{Laplace is isomorphism}, we study  the spaces $C^{k,\alpha}_\delta (E)$ further. In fact,
  Lemma \ref{weight function equivalent to distance} allows us to obtain the expected embedding theorems and also Schauder estimates for $\Delta$.

 \begin{lemma}[Embeddings]\label{embeddings of holder spaces}
 	 Let $k,l \in \mathbb{N}$, $0<\alpha_0,\alpha_1<1$ and $\delta_0\leq \delta_1$. Then there are the following continuous embeddings:
 	 \begin{itemize}
 	 	\item[(i)] $C^{k} _{\delta_0} (E) \subset C^{l} _{\delta_1}(E)$ if $l\leq k$,
 	 	
 	 	\item[(ii)] $C^{k,\alpha_0} _{\delta_0}(E) \subset C^{l,\alpha_1} _{\delta_1} (E)  $ if $ l\leq k$ and $\alpha_1\leq \alpha_0$ ,
 	 	
 	 	\item[(iii)] $C^{k+1}_\delta(E) \subset C^{k,1} _\delta(E)$.
 	In particular, 
 	$
 	C^\infty _\delta (E) = \bigcap _{k\in \mathbb{N}_0} C^{k,\alpha}_\delta (E)
 	$. 
 	
 	 \end{itemize}
 \end{lemma}
The proof of this lemma is analogue to \cite{chaljub1979problemes}[Lemma 2], so we omit it here.

Now we are in a position to show Proposition \ref{Laplace is isomorphism}. 

\begin{proof}[Proof of Proposition \ref{Laplace is isomorphism}]
 	Part (i) is a direct consequence of  Theorem 1.5 in  \cite{hein2011weighted}. Indeed, $(E,g_\varphi)$ is SOB($k+1$) by Lemma \ref{lemma check SOB}, and, because of Lemma \ref{weight function equivalent to distance}, we have that $|h|=O(\rho^{-\delta})$ with $\delta>2$, where $\rho(x)$ denotes the distance to some fixed point $x_0$. Then  (i) is precisely \cite{hein2011weighted}[Theorem 1.5].

	For (ii), we note that the function $w$ satisfies the assumption of Theorem 1.6 in \cite{hein2011weighted}, see Lemma \ref{weight function equivalent to distance}. 
 Consequently, \cite{hein2011weighted}[Theorem 1.6] gives a $u \in C^{2,\alpha}(E)$ such that $\Delta u=h$ and $u=O(w^{2-\delta + \varepsilon})$ for all $\varepsilon >0$.  
 
 Then it only remains to verify the decay rate of $|\nabla u|$, which is a consequence of standard Schauder theory. Indeed,   since the curvature of $g_\varphi$ is bounded by Lemma \ref{curvature of g varphi is bounded},  we can find $s>0$ and $Q>0$ such that for all $x\in E$, there is a chart $\Phi_x$ from to the Euclidean ball $B_x(s)\subset \mathbb{R}^{m+d}$ of radius $s$ onto a neighborhood of $x$ so that $\frac{1}{Q}  g_{\text{euc}}\leq \Phi_x^* g_\varphi \leq Q g_{\text{euc}}$ and $|| \Phi_x ^* g_{\varphi} ||_{C^{1,\alpha} (B_x(s))} \leq Q$  (\cite{petersen1997convergence}[Theorem 4.1]).  Here, $g_{\text{euc}}$ denotes the flat metric and $||\cdot ||_{C^{1,\alpha}(B_x(s))}$ the Euclidean H\"older norm. For simplicity,  we suppress the chart $\Phi_x$ and view $B_x(s)$ as a subset of $E$. Also note that we can assume that $s$ is strictly smaller than  the injectivity radius of $(E,g_\varphi)$. 
 Applying the Euclidean Schauder estimates (\cite{gilbarg2015elliptic}[Theorem 6.2]) to the balls $B_x(s)$, we obtain that 
 \begin{align}\label{local schauder estimates for controll}
 	|\nabla u | _{g_\varphi} (x)&\leq Q |du|_{g_{\text{euc}}} (x) \nonumber\\
 	&\leq Q ||u||_{C^{2,\alpha} (B_x(s))} \\
 	&\leq QC_0 \left(  ||h||_{C^{0,\alpha} (B_x(s))} + ||u||_{C^0 (B_x(s))} \right) \nonumber
 \end{align}
 for  some uniform constant $C_0>0$ depending only on $\alpha$, $s$ and $Q$. Moreover, the weight function $w$ is chosen so that there is a uniform constant $C_1>0$ such that for all $x\in E$ with $t(x)\gg 1$ and all $y\in B_x(s)$, we have $\frac{1}{C_1} w(y)\leq w(x) \leq C_1 w(y)$. This follows directly from Lemma \ref{weight function equivalent to distance} and the fact that $g_\varphi$ and $g_{\text{euc}}$ are uniformly equivalent on $B_x(s)$. Therefore, we continue to estimate for all $x\in E$ with $t(x)\gg 1$  and all $y\in B_x(s)$:
 \begin{align*}
 	u(y) \leq C w(y)^{2-\delta + \varepsilon} \leq C C_1 ^{2-\delta+\varepsilon} w(x)^{2-\delta+\varepsilon},
 \end{align*}
i.e. $||u||_{C^0(B_x(s))} =O(w(x)^{2-\delta+\varepsilon})$. Similarly, we conclude that 
\begin{align*}
	||h||_{C^{0,\alpha}(B_x(s))} =O(w(x)^{-\delta})
\end{align*}
because $s$ is chosen  strictly  smaller than the injectivity radius and $h\in C^{0,\alpha}_{-\delta}(E)$. In combination with (\ref{local schauder estimates for controll}) we consequently arrive at
\begin{align*}
	|\nabla u |_{g_\varphi} = O(w^{2-\delta + \varepsilon}),
\end{align*}
as claimed. 
\end{proof}

The issue with (ii) of Proposition \ref{Laplace is isomorphism} is that one can in general not conclude  $u=O(w^{2-\delta})$  if $u= O(w^{2-\delta+\varepsilon})$ for all $\varepsilon>0$. For proving Theorem \ref{del del bar lemma}, however, we would like to  conclude that indeed $u=O(w^{2-\delta})$. The following proposition gives a criterion, when this conclusion is true.

\begin{lemma}\label{proposition improving the rate of a function}
	Let $\delta>0$ and suppose that $\xi \in C^{\infty}_{-1-\delta}(T^*E)$. If $\xi= du$ for some $u \in C^1(E)$, then there exists a constant function $u_c$ such that $u-u_c \in C^{\infty}_{-\delta}(E)$. If additionally $u \to 0$ as $t\to \infty$, then $u_c\equiv0$.  
\end{lemma}
\begin{proof}
	This lemma is proven analogously to the corresponding statement for  conical metrics \cite{marshallphd}[Lemma 5.10]. First observe that we only need to find a constant $u_c$ such that $u-u_c \in C^{0}_{-\delta}(E)$ because $\nabla (u- u_c) = du \in C^{\infty}  _{-1-\delta} (T^*E)$ by assumption.

	We work on $E\setminus D \cong \mathbb{R} \times S$ and fix a point $(t_0,y_0) \in \mathbb{R} \times S$. Viewing $S$ as the slice $\{ 0\}\times S$, we endow $S$ with a metric $g_S$ by restricting $g_\varphi$ to $S$. For a different point $(t,y)$, we let $q_{t_0,t}$ be the straight line path from $(t_0,y_0)$ to $(t,y_0)$ and $q_{y_0,y}$ be a path joining the points $
	(t,y_0)$ and $(t,y)$, so that its projection onto $S$ is a length minimizing geodesic. Then we have by Stoke's theorem
	\begin{align}\label{proof of lemma 5.4 using stokes}
	u(t,y)-u(t_0,y_0)= \int_{q_{t_0,t}} \xi + \int_{q_{y_0,y}} \xi.
	\end{align}
	As in the proof of \cite{marshallphd}[Lemma 5.10 (c)], the key is to notice that the integral of $ \xi$ along the path $q_{t_0,\infty}$ is finite, where $q_{t_0,\infty}$ is the linear path from $(t_0,y_0)$ to $(+\infty, y_0)$. Indeed, since $\xi  \in C^{\infty}_{-1-\delta}(T^*E)$ and $\delta>0$, we can estimate
	\begin{align}\label{first estimate of integral of xi}
	\left|\int_{q_{t_0,\infty}}  \xi\right| &\leq \int_{t_0}^{\infty} \left|\xi (\dot{q}_{t_0,\infty}) \right|ds \\ \nonumber
&\leq 	||\xi||_{C^0_{-1-\delta}} \int_{t_0}^{\infty} f'' w^{-1-\delta} ds \\
&\leq ||\xi||_{C^0_{-1-\delta}} \frac{w^{-\delta} (t_0)}{\delta}<+ \infty , \nonumber
	\end{align} 
	Splitting the integral $\int_{q_{t_0,\infty}} \xi$ into two parts, we can rewrite (\ref{proof of lemma 5.4 using stokes}) as follows:
	\begin{align}\label{RHS in Lemma}
	u(t,y) - u(t_0,y_0) - \int_{q_{t_0,\infty}} \xi = - \int _{q_t,\infty} \xi + \int_{q_{y_0,y}} \xi.
	\end{align}
	As in (\ref{first estimate of integral of xi}), it is easy to see that the right hand side of (\ref{RHS in Lemma}) is bounded by $w^{-\delta}(t)$. In fact, we have
	\begin{align} \label{estimate qy0 y}
	\left|  \int_{q_{y_0,y}} \xi \right| &\leq \int_{a}^{b} |\xi|_\varphi| \dot{q}_{y_0,y}|_{\varphi} ds\\ \nonumber
	 &\leq C ||\xi||_{C^0_{-1-\delta}} w^{-1-\delta}  (t) \sqrt{f'(t)}  \int_{a}^{b} |\dot{q}_{y_0,y}|_{g_S} ds\\
	&\leq C ||\xi||_{C^0_{-1-\delta}} w^{-\delta} (t) \operatorname{diam}(S,g_S),\nonumber
	\end{align}
	where $q_{y_0,y}$ is defined on the interval $[a,b]$ and $C>0$ is some constant independent of $t$. Combining with (\ref{first estimate of integral of xi}), we obtain
	\begin{align*}
\left| u(t,y) - u(t_0,y_0) - \int_{q_{t_0,\infty}} \xi \right| \leq ||\xi||_{C^0_{-1-\delta}} \left( \delta^{-1} + C \operatorname{diam}(S,g_S)   \right) w^{-\delta}(t),
	\end{align*}
	i.e. $u-u_c \in C^{0}_{-\delta}(E)$ where we set  
	\begin{align*}
	u_c&= u(t_0,y_0) + \int_{q_{t_0,\infty}} \xi \\
	& = u(t_0,y_0) + \lim_{t\to \infty} \left( u(t,y_0)- u(t_0,y_0)   \right)\\
	&=\lim_{t\to \infty} u(t,y_0).
	\end{align*}
	Thus, it remains to show that $u_c$ is indeed constant. Let $q_{y_0,y_1}$ be a path in the slice $\{t\} \times S$ connecting two points $(t,y_0) $ and $ (t,y_1)$. Then we have
	\begin{align}\label{limit independent of y}
	u(t,y_1)-u(t,y_0)= \int_{q_{y_0,y_1}} \xi, 
	\end{align}
	and by (\ref{estimate qy0 y}), the right hand side of (\ref{limit independent of y}) goes to $0$ as $t\to \infty$. Hence $\lim_{t\to \infty} u(t,y_0)= \lim_{t \to  \infty } u(t,y_1)$ for any $y_0,y_1 \in S$, proving the lemma. 
\end{proof}

\subsection{Vanishing of harmonic forms. } \label{section vanishing theorem}
We aim at proving a vanishing theorem for harmonic  (1,0)-forms on the manifold $(E,g_\varphi)$. This will be needed for the $\partial \partialb$-Lemma. We start with a basic observation which is immediate from the standard Bochner formula. 

\begin{lemma}\label{vanishing lemma using ric geq 0}
	Any harmonic 1-form $\beta$ on $(E,g_\varphi)$ such that $|\beta| \to 0$ as $t\to \infty$ must vanish identically. 
\end{lemma}

\begin{proof}
	Since $\operatorname{Ric}(\omega_{\varphi})$ is non-negative by Theorem \ref{sign of ricci curvature Thm}, the Bochner formula reads
	\begin{align*}
		\Delta |\beta |^2 \geq 0 ,
	\end{align*}
	and the claim then follows from the Maximum principle.
\end{proof}

It becomes more interesting if we replace the asymptotic condition of $\beta$ in the previous lemma by requiring that $\beta$ be square-integrable instead. If $\beta$ is moreover of type $(1,0)$, it is also holomorphic and it must be zero by the following Theorem.

\begin{theorem}\label{vanishing theorem}
	Any  $L^2$-holomorphic (1,0)-form on $(E,g_\varphi)$ is identically zero. 
\end{theorem}

\begin{proof}
	We adapt the idea behind \cite{munteanu2015kahler}[Theorem 7]. Let $\beta$ be a holomorphic (1,0)-form, which is square integrable w.r.to the metric $g_\varphi$. Then $\partialb \beta = \partialb^* \beta =0$, and by the K\"ahler identities $\Delta_d \beta =0$, i.e. $\beta$ is harmonic. Since every  $L^2$-harmonic form on a complete manifold is closed and coclosed, we conclude $d\beta =d^* \beta=0$. Observe that $\beta$ and $\pi^*  j^* \beta$ are in the same de-Rham cohomology class, where $\pi : E \to D$ is the projection and  $j: D \to E$ is the inclusion of $D$ as the zero section. Hence $\beta= \pi ^* j^* \beta + \partial h$ for some function $h$. It follows immediately that $\partialb \partial h=0$. 
	
	For some $\varepsilon>0$, consider  the tube $D_\varepsilon=\{ z\in E \;|\; r(z)\leq \varepsilon  \}$ around the zero section. Then by Stoke's theorem, there is the following formula
	\begin{align}\label{vanishing theorem integration formula}
	\int _{D_\varepsilon} |\partial h|^2 = - \int _{D_\varepsilon} \langle  h,\partial ^*\partial  h  \rangle + \int_{\partial D_\varepsilon} h \iota _{\nu} (\partial h ).  
	\end{align}
	Here, $\nu:= \frac{X}{|X|}$ denotes the outward pointing unit normal vector to $\partial D_\varepsilon$. As $X$ is a real holomorphic vector field, the function $\iota_X (\partial h)$ is also holomorphic and we claim that it is  in $L^2$. Indeed, using $\iota_X (\pi^* j^* \beta)=0$, we observe that 
	\begin{align*}
		|\iota_X (\partial h)| = |\iota_X (\beta)| \leq |X|\cdot |\beta|
	\end{align*}
	so that $\iota_X(\partial h)$ is square-integrable since $X$ is bounded and $\beta $ is  $L^2$. Hence, $\iota_X (\partial h)$ is an $L^2$-holomorphic function and must consequently be zero.
	 Moreover, $2 \partial^* \partial h= \Delta h =0$ because $h$ is pluriharmonic. Thus,  $\partial h$ vanishes identically on $D_\varepsilon$ by (\ref{vanishing theorem integration formula}). So $\partial h$ must be zero everywhere since it is a holomorphic (1,0)-form.
	
	 We conclude that $\beta = \pi ^* j^* \beta$.
	However, a form pulled back from the base can never be in $L^2$, unless it vanishes identically. Indeed, let $\alpha$ be a 1-form on $D$ which is non-zero at some point $p$.  Keeping the expression (\ref{metric g varphi in coordinates}) in mind,  we can always estimate in a neighborhood around $p$
	\begin{align*}
		\langle \pi ^* \alpha , \pi ^*\alpha \rangle \geq C w^{-1} >0
	\end{align*}
	for some  $C>0$ independent of $t$. It follows that $\int _E |\pi^* \alpha|^2 = + \infty$ since $w^{-1}$ is not integrable. This finishes the proof. 
\end{proof}

\subsection{The $\partial \bar \partial$-Lemma.} \label{section del del bar lemma} In this paragraph, we prove Theorem \ref{del del bar lemma}  on the manifold $E$ analogue to \cite{HAJOricciflat}[Theorem 3.11].

The first step is to find a primitive of $\eta$, with controlled growth. In fact, one can write down an explicit primitive for $\eta$ on the product $E\setminus D\cong \mathbb{R} \times S$ and then read off its growth behaviour. This is the idea behind the next proposition.

\begin{prop}\label{prop primitive with growth }
	Let $\delta >2$ and $\eta\in C^{\infty}_{-\delta}(\Lambda^2 T^*E)$ be a $d$-exact $2$-form. Then $\eta= d\theta$ for some $\theta \in C^{\infty}_{-\delta +1}(T^* E)$. 
\end{prop}

 \begin{proof}
 	As in \cite{HAJOricciflat}[Theorem 3.11], we first reduce the problem to finding a primitive for $\eta $ on the product $\mathbb{R} \times S$. By assumption, there exists a $1$-form $\xi$ such that $\eta= d\xi$. Let $t_1<t_2$ and define two compact sets $K_j$ with $j=1,2$ by
 	\begin{align*}
 	K_j= \{ z\in E \:|\: t(z)\leq t_j  \},
 	\end{align*}
 	where we view the zero section of $E$ to be the set $\{z\in E \: |\: t(z)=-\infty \}$.
We pick a cut-off function $\chi$ so that $\chi \equiv 0$ on $K_1$ and $\chi\equiv 1$ on the complement of $K_2$. Then we put $\hat{\xi}:= \chi \xi$ and $\hat{\eta}:= d \hat{\xi}$. Note that if $\hat{\eta} = d\hat{\theta}$ for some $\hat{\theta}$, then $\theta:=\xi - \hat{\xi } + \hat{\theta}$ satisfies 
\begin{align*}
d\theta = d\xi - d( \chi \xi) + \hat{\eta} = \eta.
\end{align*}
Since $\theta=\hat{\theta}$ outside $K_2$, it suffices to find $\hat{\theta}\in C^{\infty}_{1-\delta}(\Lambda^* E)$ with $\hat{\eta}= d\hat{\theta}$ and $\hat{\theta}\equiv 0$ on $K_1$. 
The following construction of $\hat\theta$ can be found in the proof of \cite{marshallphd}[Proposition 5.8]. 

For each $t\in \mathbb{R}$, there is an inclusion $i_{t} : \{t \}\times S \to \mathbb{R}\times S$ given by $i_t (y)=(t,y)$. Write $\hat{\eta}= dt\wedge \hat{\eta}_1 + \hat \eta _ 2$, where $\hat \eta _j $ are 1-parameter families of $j$-forms such that 
\begin{align}\label{vanishing of eta j}
\iota_{\partialt}  \hat{\eta} _j =0 \;\;\; \text{and}  \;\;\; i^*_t \hat{\eta}_j=0 \;\;\; \text{for all } t\leq t_1.
\end{align}
We define a family $\hat{\theta}_t$ with $t\in \mathbb{R}$  of 1 forms on $S$ by
\begin{align} \label{theta hat definition}
\hat{\theta}_t = -\int_{t}^{\infty} i_s^*(\hat{\eta}_1)ds.
\end{align}
Then we  define a 1-form $\hat{\theta}$ on $\mathbb{R} \times S$ by requiring that
\begin{align} \label{definition of theta from theta t}
\iota_{\partialt} \hat{\theta} =0 \;\;\; \text{and }\;\; i^*_t \hat{\theta} = \hat{\theta}_t \;\;\; \text{for all } t\in \mathbb{R}.
\end{align}
We have to prove that $\hat \theta$ is well-defined, i.e. that the integral (\ref{theta hat definition}) exists. We start by looking at $|\hat{\eta}_1|$. As  $dt$ and $\hat{\eta}_1$ are orthogonal to each other, we have that 
\begin{align*}
|dt \wedge \hat{\eta} _1 | = |dt| |\hat{\eta}_1|  = \frac{1}{\sqrt{f''(t)}} |\hat{\eta}_1| .
\end{align*}
Using  that $dt\wedge \hat{\eta} _1$ is orthogonal to $\hat{\eta}_2$, we can estimate
\begin{align}\label{estimate of eta hat 1}
|\hat{\eta} _1 | = \sqrt{f''(t)}|dt \wedge \hat{\eta}_1| \leq \sqrt{f''(t)} |dt\wedge \hat{\eta}_1 + \hat{\eta}_2| = O(w^{-\delta}),
\end{align}
since $f''$ is bounded and $|\hat{\eta}|=O(w^{-\delta})$ by assumption. To compute the integral (\ref{theta hat definition}), we  work in coordinates. Let  $(y_0=t,y_1,\dots , y_{2(m+d)-1})$ be real coordinates of $\mathbb{R}\times S$ and write $\hat{\eta }_1 = \sum _{j\geq 1} \hat \eta _{1,j} dy_j$. Then (\ref{theta hat definition}) becomes
\begin{align}\label{theta integral in coordinates}
\hat \theta _t = - \sum _{j \geq 1} \left(\int _t ^{\infty} i^*_s \hat \eta _{1,j} ds \right) dy_j .
\end{align}
Note that the norms $|dy_j|$ may not have the same asymptotic behaviour for different values of $j= 1,\dots ,m+d-1$. In fact, it follows from (\ref{metric g varphi in coordinates}) that we have 
\begin{align}
  |dy_j|= \begin{cases*}
O(w^{-\frac{1}{2}}) & if $\pi^*\hat{g}_{jj}>0$,\\
O(1) & if $\pi ^*\hat{g}_{jj}=0$, \\
\end{cases*}
	& \text{and }\; \frac{1}{|dy_j|}= 
	\begin{cases*}
	O(w^{\frac{1}{2}}) & if $\pi ^* \hat{g}_{jj}>0$,\\
	O(1) & if $\pi ^* \hat{g}_{jj}=0$. \\
	\end{cases*}
\end{align}
As $|\hat{\eta}|=O(w^{-\delta})$, we conclude that either $|\hat \eta _{1,j}|=O(w^{-\delta +\frac{1}{2}})$  or $|\hat \eta _{1,j}|=O(w^{-\delta})$ and hence, the integrals in (\ref{theta integral in coordinates}) are all finite because we chose $-\delta+1<-1$. 

We also observe from (\ref{vanishing of eta j}) that  $\hat{\theta} _t =\hat{\theta}_s $ for all  $s,t\leq t_1$, so $\hat{\theta}$ extends to a smooth 1-form on $E$. 
Moreover, we can read
 off (\ref{theta integral in coordinates})  that $| \hat\theta|  =O(w^{-\delta +1})$, i.e. $\hat\theta \in C^{0}_{-\delta +1}(T^*E)$. 
It is possible to obtain estimates on derivatives of $\hat{\theta}$ and to show that $\hat{\theta} \in  C^{\infty}_{-\delta +1}(T^*E)$. However, this is a long calculation which relies only on two main observations. First, we deduce from
Lemma \ref{growth of derivatives of f} that $|\nabla^l dy_j|$ behaves asymptotically like $|dy_j| w^{-l}$ for all $l\geq 0$ and $j =0,\ldots , 2(m+d)-1$. Secondly, we can conclude from $\eta \in C^{\infty}_{-\delta } (\Lambda^*T^*E)$ that also $|\nabla^l \hat{\eta }_1|=O(w^{-\delta -l})$. Using formula (\ref{theta integral in coordinates}), it is then straight forward to verify $|\nabla^l\hat{\theta}| =O(w^{-\delta -l +1})$, as claimed.  We leave the details to the reader, but remark that the required estimate is similar to bounding $|\hat{\theta}|$.

It remains to show  that $\hat{\eta}= d\hat{\theta} $ by considering its components. In fact, it is an easy calculation (\cite{marshallphd}[p.80]) to prove that 
\begin{align*}
 \partialt \left(i^*_t (\hat{\eta} - d\hat{\theta})  \right)=0,
\end{align*}
i.e. $i_s^*( \hat{\eta} - d \hat{\theta} ) =i_t ^* (\hat{\eta} - d \hat{\theta} )$ for all $s,t \in \mathbb{R}$. Since $\hat{\eta}, \hat{\theta} \to 0$ as $t\to \infty$, we conclude that $i_t^*(\hat{\eta} - d \hat{\theta}) =0$ for any $t\in \mathbb R$. Moreover, it is shown in \cite{marshallphd}[p.80] that 
\begin{align*}
\iota_{\partialt} \hat{\eta} = \iota_{\partialt} d \hat{\theta},
\end{align*}
and hence  $\hat{\eta} = d\hat{\theta}$ as we claimed.

 \end{proof}

\begin{proof}[Proof of Theorem \ref{del del bar lemma}]
	The strategy is the same as for the proof of \cite{HAJOricciflat}[Theorem 3.11].
	We start with some basic observations. By Proposition \ref{prop primitive with growth },  there exists a $\theta \in C^\infty _{1-\delta }(\Lambda^* E)$ such that $d\theta= \eta$. Since $\eta$ is real, $\theta$ will also be a real form, i.e. $\theta^{1,0}= \overline{\theta^{0,1}}$ if $\theta= \theta^{1,0}+ \theta^{0,1}$ is the decomposition into types.
	Moreover,  $\eta$ is of type (1,1), so we must have that $\partial \theta ^{1,0}= \partialb \theta ^{0,1}=0$.
	
	If $\partial^*$ denotes the formal dual of $\partial$ (w.r.to the $L^2$-metric induced by $g_\varphi$), then $\partial ^* \theta ^{1,0} \in C^{\infty}_{-\delta }(E)$. We would like to find a solution $u$ to the equation $\Delta u = \partial ^* \theta^{1,0}$, whose growth we can control. There are two cases to consider, corresponding to part ($i$) and $(ii)$ of Proposition \ref{Laplace is isomorphism}.
	
	First, we consider the case where the degree $k$ of the polynomial $Q$  is greater or equal to $2$. By ($ii$) of Proposition \ref{Laplace is isomorphism}, there exists $u \in C^{2,\alpha}(E)$ such that $\Delta u= \partial^* \theta^{1,0}$ and $|u|+ |\nabla u |=O(w^{2-\delta +\varepsilon})$. It follows that $\partial ^*\left(\partial u - \theta^{1,0}\right) = \partial \left( \partial u - \theta^{1,0}\right) =0$, and hence the 1-form $\partial u - \theta^{1,0}$ is harmonic by the K\"ahler identities.

	 Choosing $\varepsilon>0$ small enough, we can assume that $2-\delta + \varepsilon<0$ and hence we see from $|\nabla u|=O(w^{2-\delta + \varepsilon})$  and $\theta \in C^{\infty}_{1-\delta}(\Lambda^* E)$ that
	 \begin{align*}
	|\partial u- \theta^{1,0}| \leq |d u | + |\theta| \to 0,  
	\end{align*}
	as $t \to \infty$. Then Lemma \ref{vanishing lemma using ric geq 0}  implies $\partial u - \theta^{1,0}= 0$ and consequently,
	\begin{align*}
	\eta = d\theta = \partial \theta ^{0,1} + \partialb \theta ^{1,0}= \partial \partialb  \bar u + \partialb \partial u = -2 \sqrt{-1} \partial \partialb \operatorname{Im} u ,
	\end{align*}
	where $\operatorname{Im}u$ is the imaginary part of $u$. 
	
	It remains to show that $\operatorname{Im} u \in C^{\infty} _{2-\delta}(E)$ as opposed to only $\operatorname{Im}u \in C^{2,\alpha}(E)$ and $\operatorname{Im}u =O(w^{2-\delta +\varepsilon})$. As we can choose $\varepsilon>0$ so that $2-\delta+ \varepsilon<0$, this improvement of the decay rate, however, follows immediately from Proposition \ref{proposition improving the rate of a function} if we can show $d \operatorname{Im} u \in C^{\infty}_{1-\delta}(\Lambda^*E)$. This last condition is clearly satisfied  since $\theta^{1,0}, \theta^{0,1} \in C^{\infty}_{1-\delta}(\Lambda^*E)$ and $\theta^{1,0}- \theta ^{0,1}=\partial u - \overline{\partial u}= d\operatorname{Re} u + \sqrt{-1}d \operatorname{Im}u$. This settles the first case.
	
	For the second case, assume that $k\leq 1$. We want to use (i) of Proposition \ref{Laplace is isomorphism} to solve $\Delta u = \partial ^* \theta^{1,0}$. This time, however, we only know that the solution $u$ satisfies $\int |\nabla u|^2 \omega_{\varphi}^{m+d} <+ \infty$, and not necessarily that $u$ decays towards infinity. So the idea is to use the vanishing Theorem \ref{vanishing theorem} instead.
	
	Before applying Proposition \ref{Laplace is isomorphism} (i), we need to verify that $\int \partial ^* \theta^{1,0} \omega_{\varphi}^{m+d} $ is zero. For any $t_0\in \mathbb{R}$, define
	$K_{t_0}=\{ z\in E \: |\: t(z)\leq t_0  \}$ and consider the integral
	\begin{align}\label{integral of partial star theta is zero}
	\int_{K_{t_0} }\partial ^* \theta^{1,0} \omega_{\varphi}^{m+d} = \int_{K_{t_0}} d * \theta ^{1,0} = \int_{\{t_0\}\times S} * \theta ^{1,0},
	\end{align}
	where we used Stoke's for the last equality. If we equip the slice $\{t_0\} \times S$ with  the restriction of $g_\varphi$ and denote the corresponding volume by $\operatorname{Vol}(\{t_0\}\times S)$, then we can estimate
	\begin{align*}
	\left| \int_{\{t_0\} \times S} * \theta ^{1,0} \right| \leq  \operatorname{Vol}(\{t_0\}\times S) \sup _{\{t_0\} \times S} |\theta| = O(w^{k+1-\delta} (t_0)),
	\end{align*}
	since $|\theta|= O(w^{1-\delta})$ and $\operatorname{Vol}(\{t_0\}\times S)=O(w^{k})$. It follows that the right hand side of (\ref{integral of partial star theta is zero}) goes to zero as $t_0 \to +\infty$, as we assumed $k\leq 1$ and $\delta>2$. Hence $\int \partial ^* \theta ^{1,0} \omega_{\varphi}^{m+d} =0$, as claimed. 
	
	So we find a $u\in C^{2,\alpha}(E)$ such that $\Delta u = \partial ^* \theta^{1,0}$ and $\int |\nabla u|^2 \omega_{\varphi}^{m+d} $ is finite. In particular, the 1-form $\beta= \theta^{1,0} - \partial u $ is harmonic. Also note that  
	\begin{align*}
	|\theta|\omega^{m+d}_\varphi = O(w^{2-2\delta+ k } ) 
	\end{align*}
	with $2-2\delta+k<-1$, so that $\theta$ is in $L^2 $, and thus $\beta$ is $L^2 $ as well. 
	
	It follows that $d\beta = d^* \beta =0$, and in particular, $\beta$ is an $L^2$-holomorphic (1,0)-form. Hence it must be identically zero by Theorem \ref{vanishing theorem}, i.e. $\theta^{1,0} = \partial u$. The rest of the proof is now analogous to the first case.

\end{proof}

	\bibliography{ms}
	\bibliographystyle{amsalpha}
\end{document}